\documentclass{amsart}

\usepackage{amssymb}
\usepackage[cmtip,all]{xy}

\begin{document}

\newtheorem{theorem}{Theorem}[section]
\newtheorem{lemma}[theorem]{Lemma}
\newtheorem{proposition}[theorem]{Proposition}
\newtheorem{conjecture}[theorem]{Conjecture}
\newtheorem{corollary}[theorem]{Corollary}

\newtheorem*{wwc}{The weak generalised Weinstein conjecture}
\newtheorem*{swc}{The strong generalised Weinstein conjecture}

\theoremstyle{definition}
\newtheorem{definition}[theorem]{Definition}
\newtheorem{example}[theorem]{Example}
\newtheorem{xca}[theorem]{Exercise}

\theoremstyle{remark}
\newtheorem{remark}[theorem]{Remark}

\numberwithin{equation}{section}

\title{Contact foliations and generalised Weinstein conjectures}

\author{Douglas Finamore}
\address{Departament of Mathematics - ICMC, University of S\~ao Paulo}
\email{douglas.finamore@usp.br}
\thanks{This work was funded by Brazilian Coordena\c c\~ao de Aperfei\c coamento de Pessoal de N\'ivel Superior (CAPES), grant PROEX-11377206/D. I am deeply grateful to my advisor, Prof. Carlos Maquera, for all the advice and useful insight provided during the course of my graduation, of which this work is a direct result.}



\dedicatory{to Camila}

\begin{abstract}

We consider contact foliations: objects which generalise to higher dimensions the flow of the Reeb vector field on contact manifolds. We list several properties of such foliations and propose two conjectures about the topological types of their leaves, both of which coincide with the classical Weinstein conjecture in the case of contact flows. We give positive partial results for our conjectures in particular cases -- when the holonomy of the contact foliation preserves a Riemannian metric, for instance -- extending already established results in the field of Contact Dynamics.

\end{abstract}

\maketitle

\section{Introduction}
\par The concept of $q$-contact structure, as considered in this paper, first appeared in Almeida's doctoral dissertation \cite{almeida_contact_2018}. It was conceived having in mind the problem of classification of Anosov actions of $\mathbb{R}^q$ and the Verjovsky Conjecture for actions (c.f. \cite{barbot_integrable_2011}). The idea was to define a structure that could replicate the proprieties of Anosov contact flows in higher dimensions. 
Basically, a $q$-contact structure consists of a collection of $q$ linearly independent $1$-forms $\lambda_1, \cdots, \lambda_q$ with the property that all the derivatives $d\lambda_i$ are non-degenerate in the intersection of their kernels.
In Almeida's work, particular focus was given to the \textit{contact action} of $\mathbb{R}^q$ induced by such an object on its ambient manifold and to matters revolving around the algebraicity of said action. In this paper, we expand upon the basic theory laid down in \cite{almeida_contact_2018} by focusing not on the contact action but on its underlying \textit{contact foliation}. This foliation can be seen as a direct generalisation to higher dimensions of the Reeb flow (also called contact flow) defined by a co-orientable contact structure. We are particularly interested in the dynamics of contact foliations, especially the existence of closed orbits.

\par The most famous question in the field of Contact Dynamics is the Weinstein Conjecture ($\mathrm{WC}$) \cite{weinstein_hypotheses_1979}. It states:
\begin{center}
    \textit{
     The Reeb vector field of every closed contact manifold admits a closed orbit.
    }
\end{center}
Though not yet proven in its full generality, the conjecture is known to hold in several particular cases: in dimension three \cite{taubes_seibergwitten_2007, hutchings_taubess_2009}; for PS-overtwisted contact forms \cite{albers_weinstein_2009} and, more generally, for overtwisted contact manifolds \cite{borman_existence_2015}; and finally, when the Reeb vector field is Killing for some metric on $M$ \cite{rukimbira_remarks_1993, banyaga_characteristics_1995}. Given the analogy between contact foliations and Reeb flows, it is reasonable that one asks oneself whether or not some Weinstein-like results exist for contact foliations as well. There are two straightforward generalisations to higher dimensions:
\begin{wwc}[$\mathrm{WGWC}$]
    Every $q$-dimensional contact foliation on a closed manifold has a leaf homeomorphic to $\mathbb{T}^l \times \mathbb{R}^{q-l}$, for some $l \geq 1$.
\end{wwc} 
\begin{swc}[$\mathrm{SGWC}$]
    Every $q$-dimensional contact foliation on a closed manifold has a leaf homeomorphic to $\mathbb{T}^q$.
\end{swc}

\par In the one-dimensional case (Reeb flows), all three conjectures are the same. We will show that, under special conditions on the $q$-contact structure, the $\mathrm{WGWC}$ and the $\mathrm{WC}$ are equivalent (cf. Theorem \ref{equivconjuniform}). It is currently an open question whether other equivalences exist under other conditions. Moreover, if any of the conjectures hold, then there can be no \textit{minimal} contact foliations, that is, contact foliations where all the leaves as dense subsets. This is clear for the $\mathrm{WC}$ and $\mathrm{SGWC}$. For the $\mathrm{WGWC}$, it will be a straightforward consequence (cf. Theorem \ref{equivconjmin}) of a primarily algebraic construction, carried out in Proposition \ref{contredux}, which allows us to reduce a minimal contact action on a closed manifold $M$ to a contact action \textit{by planes} on a manifold $M_0$, over which $M$ is a principal torus bundle.

\par Since the closure of any leaf always contains a minimal set, investigating leaf closures might lead to a positive answer to the strong and weak conjectures or eventually to a counter-example to them. For this to work, we need to ensure that the restriction of a contact foliation to an invariant set is also a \textit{contact} foliation. One sufficient condition for that to happen is that the invariant subset is a \textit{manifold of even codimension} in $M$ (cf. Lemma \ref{restevencod}).

\par This brings our interest to a class of foliations that is very well-behaved concerning leaf closures: Riemannian foliations. It turns out that for contact foliations, being Riemannian is equivalent to being isometric, i.e., the orbit foliation of an $\mathbb{R}^q$-action \textit{via isometries} (cf. Proposition \ref{Rcontacttransverseproperty}). It follows then, from classical results from Myers-Steenrod \cite{myers_group_1939} and Kobayashi \cite{kobayashi_fixed_1958}, that for Riemannian foliations the leaf closures have the necessary properties: they are submanifolds of even codimension. Following the steps laid down by Banyaga and Rukimbira in \cite{banyaga_characteristics_1995}, we can show that Riemannian contact foliations satisfy the $\mathrm{WGWC}$. Going even further, a reduction argument can be used to prove

\begin{theorem}\label{riemanniancase}
    On a closed manifold, the $\mathrm{SGWC}$ is satisfied by every Riemannian contact foliation.
\end{theorem}

\par With the isometric case settled, one may ask what happens in the complete opposite situation, i.e., when the contact foliation is also \textit{Anosov}, so that instead of preserving distances, the holonomy transformations contract and expand simultaneously in different transverse directions. Since contact foliations first appeared in the setting of Anosov actions, it is only natural to ask if this class of foliations satisfy our conjectures. It turns out the problem is rather uninteresting, as the Anosov property alone almost implies the existence of closed orbits \cite[Theorem 5]{barbot_transitivity_2011}. Hence, every Anosov contact foliation on a closed manifold (in particular, the ones studied in \cite{almeida_contact_2018}) satisfy the $\mathrm{SGWC}.$

\subsection{Structure of the paper}
In Section \ref{contactfol}, we give definitions and examples and expand the work started in \cite[Section 3.2.1]{almeida_contact_2018} by proving several properties of contact foliations. In particular, when studying their holonomies, a new object -- which does not occur in the classical contact case -- appears: the characteristic foliation associated with each form $\lambda_i$. It is the orbit foliation of an $\mathbb{R}^{q-1}$-action entirely determined by $\lambda_i$ and the contact foliation, and its transversals become contact manifolds when equipped with the restriction of $\lambda_i$. In this section, we also introduce the two generalisations of the Weinstein conjecture and start pointing out scenarios where they hold or are equivalent to one another.

\par  In the second section, we also investigate \emph{uniform} $q$-contact structures, which are the ones such that all the derivatives $d\lambda_i$ are the same. Such objects were first studied by Bolle \cite{bolle_condition_1996}, who characterised them as compact submanifolds of a symplectic manifold satisfying what he calls the ``$C_p$ condition''. In short, Bolle's $C_p$ condition generalises the concept of a contact hypersurface to submanifolds of higher codimension. Our approach has the advantage of being intrinsic, without the need for the ambient symplectic manifold. 
The existence of a uniform $q$-contact structure imposes severe restrictions on the topology of the manifold on which it is defined. For instance, using Tischler's argument, we show that every such manifold fibres over the torus $\mathbb{T}^{q-1}$ (cf. Theorem \ref{uniformStructuresFibrations}).  
More than that, we prove that for uniform $q$-contact structures, $q \geq 2$, every characteristic foliation is a suspension of an $\mathbb{Z}^{q-1}$-action on $\mathbb{T}^{q-1}$, which is interesting because no contact action can be a suspension, so in a sense, the characteristic foliations of a uniform $q$-contact structure are the orbit foliations of ``maximal'' suspended sub-actions. Moreover, as a corollary to the fact that characteristic foliations are suspensions, we obtain equivalence between the $\mathrm{WC}$ and the $\mathrm{WGWC}$ for uniform contact foliations (cf. Theorem \ref{equivconjuniform}).

\par We end with Subsection \ref{taut}, which focuses on a property of contact foliations that plays an essential role in this work: their tautness. In our context, tautness (or harmonicity) is defined as the existence of a metric tensor for which all the leaves are minimal submanifolds, i.e., submanifolds whose normal curvature vanishes. 
Note that one-dimensional minimal submanifolds are simply geodesics, so this result generalises to higher dimensions the widely known fact that contact flows are geodesible. 

\par The last section is Section \ref{partres}, dedicated to partial results towards a positive answer to the $\mathrm{SGWC}$. We briefly prove that it is satisfied by Anosov flows and then develop a theory for \textit{isometric} contact foliations: ones that preserve a transverse Riemannian metric. Tautness is of particular relevance when considering such foliations, allowing us to conclude several facts about the topology and geometry of the ambient manifold (cf. Section \ref{geotopisocon}), which in turn provide obstructions to the existence of isometric foliations on closed manifolds. One exciting result in this spirit is Theorem \ref{curverest}, which states that the sectional curvature of any Riemannian manifold supporting a contact action by isometries must be strictly positive at least one point. This further justifies our claim that the isometric case is the opposite of the Anosov one since hyperbolicity manifests itself geometrically in negative curvature. 

\par Section \ref{partres} is heavily inspired by the work of Rukimbira and Banyaga (cf. \cite{rukimbira_remarks_1993, banyaga_characteristics_1995}), and culminates in the proof of Theorem \ref{riemanniancase}. We show more, proving that an isometric contact foliation on a closed manifold must have at least \textit{two} closed orbits and that the set of closed orbits decomposes into a disjoint union of invariant even-codimensional submanifolds (cf. Theorem \ref{closedorbitsriemanniancase}). We do so in Subsection \ref{comptoract} by studying the closure of the contact action's image $\mathrm{Im}(\mathbb{R}^q \to \mathrm{Iso}(M))$ and making extensive use of a foliation-preserving action $\mathbb{T}^q \to \mathrm{Iso}(M)$ induced by such closure.

\subsection{Some remarks on nomenclature} 
To the best of the author's knowledge, the objects of interest in this paper were first defined by Almeida in \cite{almeida_contact_2018}. In his work, Almeida refers to $q$-contact structures as \textit{generalised $k$-contact structures}. We changed the index letter from $k$ to $q$ to avoid confusion with the established notion of a K-contact structure. Moreover, Almeida's notion of $q$-contact structure, taken from \cite{bolle_condition_1996, montano_integral_2008, tomassini_contact_2008}, is what we call here a \textit{uniform $q$-contact structure} (Definition \ref{uniqcontactstructures}). Our decision to rename these concepts was based, first, on the belief that our nomenclature is less cumbersome and more informative.
\par Moreover, while preparing this paper, we stumbled upon the work of Gaset et al. \cite{gaset_contact_2020}, in which yet another object called a $q$-contact structure is defined. In the sense of Gaset et al., a $q$-contact structure is even more general than the ones present in this paper since they do not require the kernels of the derivatives of the defining $1$-forms to be all the same. In light of that, we felt it made little sense to call the structures appearing in this work ``generalised," so we decided to abandon the adjective for good. 
\par Besides the $q$-contact structures in the sense of Gaset et al., several other geometric structures are similar and/or related to the structures studied here. We want to give special mention to the notions of Contact Pair (cf. \cite{bande_contact_2005}); Multicontat Structures (cf. \cite{vitagliano_l-algebras_2015}); Pluricontact and Polycontact Structures (cf. \cite{apostolov_toric_2020, van_erp_contact_2011}). The interested reader is referred to \cite[Section 3.2.2]{almeida_contact_2018} for a more detailed comparison between all these different concepts. In this paper, we also briefly explore the concept of an \textit{$f$-structure,} in the sense of Yano, which is somewhat related but not entirely similar to that of a $q$-contact structure, especially to the particular case of uniform $q$-contact structures. In particular, many of our results here have analogues in the case of $f$-structures, like the existence of Reeb vector fields, the characterisation of uniform $q$-contact manifolds (Theorem \ref{uniformStructuresFibrations}) and also our main result, Theorem \ref{riemanniancase}.

\section{Contact Foliations}\label{contactfol}

We begin by defining what is meant by contact foliation and all the other structures involved. We expose some of the basic properties of these foliations and introduce two generalisations to the Weinstein Conjecture. 

\subsection{Definition and Examples}
A contact foliation is a generalisation of the flow of the Reeb vector field determined by a co-orientable contact structure. The analogue structure determining the contact foliation is called a $q$-contact structure, defined as follows.

\begin{definition}[\textit{$q$-contact manifolds} \cite{almeida_contact_2018}]\label{qcontactstructure}
    Let $n,q$ be positive integers and consider a $2n+q$ dimensional differential manifold $M$. A \textbf{$q$-contact structure} on $M$ is a collection $\vec{\lambda} = (\lambda_1, \cdots, \lambda_q)$ of $q$ linearly independent non-vanishing $1$-forms $\lambda_i$, together with a splitting 
    \begin{equation*}
        TM = \mathcal{R} \oplus \xi
    \end{equation*}
    of the tangent bundle, satisfying the following conditions:
    \begin{itemize}
        \item[(i)] $\xi := \cap_i \ker\lambda_i$;
        \item[(ii)] for every $i$, the restriction $d \lambda_i\rvert_\xi$ is non-degenerate;
        \item[(iii)] for every $i$, one has $\ker d \lambda_i = \mathcal{R}.$
    \end{itemize}
A manifold endowed with such structure is called a \textbf{$q$-contact manifold} and denoted by $(M, \vec{\lambda}, \mathcal{R} \oplus \xi)$, or simply by $M$ when the context permits. We call the collection $\{\lambda_i\}$ an \textbf{adapted coframe} for the $q$-contact structure, and the $q$-form
\begin{equation*}
    \lambda := \lambda_1 \wedge \cdots \wedge \lambda_q \neq 0
\end{equation*}
is called the \textbf{characteristic form}. The bundles $\mathcal{R}$ and $\xi$ are called the \textbf{Reeb distribution} and \textbf{$q$-contact distribution,} respectively.
\end{definition}
\begin{remark}
	The $q$-contact and Reeb distributions, as well as the splitting $\mathcal{R} \oplus \xi$, have the same regularity class as the forms $\lambda_i$. \emph{From now on, we will always assume the splitting is at least of class $C^1$.}
\end{remark}
\par Note that, due to the linear independence of the $\lambda_i$, $\xi$ has constant rank $2n$; therefore, condition (ii) is equivalent to $d\lambda_i^n\rvert_\xi \neq 0$, or, in other words, to $(\xi, d\lambda_i)$ being a symplectic bundle over $M$. Linear independence of the adapted coframe is equivalent to the non-degeneracy of the characteristic form. From this, we can conclude that $dM_i:= \lambda \wedge d \lambda_i^n$ is a volume form on $M$ for every $i$. In particular, $M$ is orientable. As in the contact case, we can show that the $q$-contact distribution $\xi$ is maximally non-integrable, and the Reeb distribution $\mathcal{R}$ is a parallelisable bundle of constant rank $q$.
\begin{proposition}[\cite{almeida_contact_2018}]
    For any vector field $X \neq 0$ tangent to $\xi$, there is another vector field $Y$ tangent to $\xi$ such that $[X, Y]$ is not a section of $\xi$.
\end{proposition}

\begin{remark}\label{fstructures}
In \cite{yano_structure_1982}, Yano began the study of manifolds $M$ supporting a $(1,1)$-tensor $f$ satisfying the equation $f^3+f = 0$. Such a structure induces a splitting $TM = \ker f \oplus \mathrm{Im} f$, which is said to be an \emph{$f$-structure with complemented frames} if $\ker f$ is parallelisable \cite{goldberg_normal_1970}. Given an $f$-manifold with complemented frames, additional hypotheses on the kernel $\ker f$ give rise to different geometric structures such as $\mathcal{K}$-manifolds and $\mathcal{S}$-manifolds. These last two are particular cases of $q$-contact structures as considered here (cf. \cite{blair_geometry_1970}). 

On the other hand, given a $q$-contact manifold $(M, \vec{\lambda}, \mathcal{R}\oplus \xi)$, for every defining $1$-form $\lambda_i$, the symplectic bundle $(\xi, d\lambda_i)$ admits a compatible almost complex structure $J_i$. If we define $f_i$ as equal to $J_i$ on $\xi$ and identically zero on $\mathcal{R}$, then it is clear that $f_i^3+f_i = 0$. So $M$ supports (at least) as many different $f$-structures with a complemented frame as there are different $2$-forms $d\lambda_i$. The concept of an $f$-structure with complemented frames is more general than that of a $q$-contact structure.   
\end{remark}

\begin{proposition}[\cite{almeida_contact_2018}]\label{reebfields}
    There is a unique collection of linearly independent vector fields $R_1, \cdots, R_q$ tangent to $\mathcal{R}$ satisfying, for all $i, j = 1, \cdots, q,$ the relations 
    \begin{itemize}
        \item[(i)] $\lambda_i(R_j) = \delta_{ij}$;
        \item[(ii)] $[R_i, R_j] = 0;$
        \item[(ii)] $\mathcal{R} = \mathrm{Span}\{R_1, \cdots, R_q\}.$
    \end{itemize}
\end{proposition}

\begin{remark}
    In \cite[Corollary 2.4]{cabrerizo_curvature_1990}, Cabrerizo et al. showed the existence of these fields in the case of $f$-structures using techniques other than the ones in \cite{almeida_contact_2018}. The vector fields $R_i$ have the same regularity class as the splitting $TM = \mathcal{R} \oplus \xi$. 
\end{remark}
     
\begin{definition}[\textit{Reeb vector fields}]
    The $q$ (unique) linearly independent vector fields from Proposition \ref{reebfields} are called the \textbf{Reeb vector fields} of the $q$-contact structure. 
\end{definition}
\par Each Reeb vector field preserves all the $1$-forms $\lambda_i$ and their derivatives:
\begin{align*}
    \mathcal{L}_{R_j}\lambda_i&= d(\iota_{R_j} \lambda_i) + \iota_{R_j} d\lambda_i = 0; \\
    \mathcal{L}_{R_j}d \lambda_i&= d(\iota_{R_j} d\lambda_i) + \iota_{R_j} d^2\lambda_i = 0,
\end{align*} 
and therefore, the Reeb flows preserve all of the volume forms $dM_i$, as well as the symplectic bundles $(\xi, d\lambda_i)$ over $M$. 
        
\par Pairwise commutativity assures the following action of the Euclidean space $(\mathbb{R}^q, +)$ is well-defined.
\begin{definition}[\textit{$q$-contact action}]
    Let $M$ be a $q$-contact manifold with Reeb vector fields $R_1, \cdots, R_q$. Denoting by $\phi_i^t$ the flow of $R_i$, we can define a locally free action of $\mathbb{R}^q$ on $M$ by 
\begin{align*}
	F: \mathbb{R}^q \times M &\rightarrow M\\
        	(t_1, \cdots, t_q, p) &\mapsto (\phi_1^{t_1}\circ \cdots \circ \phi_q^{t_q})(p).
\end{align*}
    This is called the \textbf{$q$-contact action on $M$} associated with the $q$-contact structure.
\end{definition}
\begin{remark}
   If the splitting $TM = \mathcal{R} \oplus \xi$ is $C^r$, then so are the Reeb fields, and therefore the action $F$ is of class $C^r$ as well.
\end{remark} 
We denote by $\mathcal{F}$ the underlying foliation of $F$, that is, the $q$-dimensional foliation of $M$ whose leaves are the orbits of $F$. 

\begin{definition}[\textit{Contact foliation}] 
    The foliation $\mathcal{F}$ underlying a contact action of $\mathbb{R}^q$ is called a \textbf{contact foliation.}
\end{definition}
It is possible for distinct adapted coframes $\vec{\lambda}$ and $\vec{\eta}$ to give rise to the same Reeb and $q$-contact distributions (cf. Lemma \ref{replem}) and, therefore, to define the same contact foliation. In this case, we say $\vec{\eta}$ is a \textbf{reparameterisation} of $\vec{\lambda}$. 

\begin{example}[Contact manifolds]
    Every contact manifold $(M, \xi)$ with a transversely orientable non-integrable distribution $\xi$ is a  $1$-contact manifold once a defining form $\lambda$ is chosen. Here, $\mathcal{R}$ is the span of the Reeb vector field, and the action $F$ is the one induced by the Reeb field's flow. We will refer to this as the \textbf{contact case.}
    
\end{example} 
  
\begin{example}[Structures on $\mathbb{R}^{2n+q}$]\label{canexm}
    There's a simple $q$-contact structure on $\mathbb{R}^{2n+q}$ with coordinates $(x_1, y_1, \cdots, x_n, y_n, z_1, \cdots, z_q)$, given by the $1$-forms 
    \begin{equation*}
        \lambda_i := d z_i + \sum_{j=1}^nx_jd y_j.
    \end{equation*}
    The Reeb vector fields are $R_i = \partial_{z_i}$, and therefore $\mathcal{R} = \mathrm{Span}(\partial_{z_1}, \cdots, \partial_{z_q})$. The contact action is translation on the last $q$ coordinates, and the contact foliation $\mathcal{F}$ consists of all planes parallel to $\{0\} \times \mathbb{R}^q$. The $q$-contact distribution is 
    \begin{equation*}
        \xi = \mathrm{Span}(\partial_{x_1}, Y_1, \cdots, \partial_{x_n}, Y_n),
    \end{equation*}
    where $Y_i :=\partial_{y_i} - x_i\sum_j\partial_{z_j}$, and the volume form $d M_i$ is the canonical volume form of $\mathbb{R}^{2n+q}$, for every $i$. 
    \par All of the above holds for 
    \begin{equation*}
        \lambda_i := d z_i + \alpha,
    \end{equation*}
    assuming $\omega = d\alpha$ is an exact non-degenerate form on $\mathbb{R}^{2n} \times \{0\}$. More generally, for a collection of such non-degenerate forms $\omega_i = d\alpha_i$, any choice $\{\alpha_{j_1}, \cdots, \alpha_{j_q}\} \in \{\alpha_1, \cdots, \alpha_l\}^q$ yields an adapted coframe
    \begin{equation*}
        \lambda_i := d z_i + \alpha_{j_i}
    \end{equation*}
    for a $q$-contact structure on $\mathbb{R}^{2n+q}$. In particular, the $d\lambda_i = \omega_{j_i}$ need not be equal.
\end{example}       

\begin{definition}[\textit{Uniform $q$-contact structures}]\label{uniqcontactstructures}
    An adapted coframe $\vec{\lambda}$ (and the $q$-contact structure it defines) is called \textbf{uniform} if it satisfies
    \begin{equation*}
    d\lambda_i = d\lambda_1
    \end{equation*}
    for all $1 \leq i \leq q$. A contact foliation $\mathcal{F}$ is said to be uniform if its adapted coframe is a reparameterisation of a uniform coframe.
\end{definition}

\begin{remark}
Both $\mathcal{K}$-structures and $\mathcal{S}$-structures (cf. Remark \ref{fstructures}) are particular cases of uniform $q$-contact structures. Uniform $q$-contact structures were also studied by Bolle in \cite{bolle_condition_1996}. He defines them as codimension $p$ submanifolds of a symplectic manifold $(W, \omega)$ satisfying what he calls ``$C_p$ condition''. The concepts are the same because every compact uniform $q$-contact manifold is a codimension $q$ submanifold of its \emph{symplectisation} satisfying the $C_q$ condition. The symplectisation of $(M, \vec{\lambda}, \mathcal{R} \oplus \xi)$ consists of the product $W := M \times \mathbb{R}^q$. Given coordinates $(x, t_1, \cdots, t_q)$ on $W$, if we equip it with the non-degenerate $2$-form 
\[
    \omega := \sum_{i=1}^q\left(dt_i\wedge\lambda_i + t_id\lambda_1\right),
\]   
then the pair $(W, \omega)$ is symplectic, and $M \times {0}$ satisfies the $C_q$ condition. It is also worth noting that not every contact foliation is uniform. The algebraic contact foliations constructed by Almeida in his doctoral thesis do not admit any uniform underlying coframe. See \cite[Proposition 3.3.5] {almeida_contact_2018} for further details. Every contact foliation constructed via an $f$-structure is uniform, as noted in Remark \ref{fstructures}.
\end{remark}
    
\par The following product-like constructions give us examples of closed $q$-contact manifolds other than closed contact manifolds.

\begin{example}[Products of contact manifolds]\label{productexample} 
    Suppose $(M_1, \alpha_1)$ and $(M_2, \alpha_2)$ are contact manifolds, with $X_i$ denoting the Reeb vector field of $\alpha_i$, and $\xi_i = \ker\alpha_i$ the associated contact structure. Then the forms
    \begin{align*}
        \lambda_1 &:= \pi_1^\ast\alpha_1 - \pi_2^\ast\alpha_2 \\
        \lambda_2 &:= \pi_1^\ast\alpha_1 + \pi_2^\ast\alpha_2 
    \end{align*}
    defined on the product $M := M_1 \times M_2$ by means of the canonical projections $\pi_i: M \to M_i$ are an adapted coframe for a $2$-contact structure on $M$. The $2$-contact distribution is 
    \begin{equation*}
    \xi := \ker\lambda_1 \cap \ker\lambda_2 = \pi_1^\ast\xi_1\cap\pi_2^\ast\xi_2,
    \end{equation*}
    on which the $d\lambda_i$ are non-degenerate due to the non-degeneracy of the $\alpha_i$. The Reeb fields are 
    \begin{align*}
        R_1 &:= \frac{1}{2}(\pi_1^\ast X_1 - \pi_2^\ast X_2), \\
        R_2 &:= \frac{1}{2}(\pi_1^\ast X_1 + \pi_2^\ast X_2).
    \end{align*}
\end{example}

\begin{example}[Flat $\mathbb{T}^l$-bundles over $q$-contact manifolds]\label{flatbundleexample}
    If $(M, \vec{\lambda}, \mathcal{R} \oplus \xi)$ is a $q$-contact manifold, the product $E := \mathbb{T}^l \times M$ supports a $(q+l)$- contact structure. More generally, if $\pi: E \rightarrow M$ is a principal $\mathbb{T}^l$-bundle supporting a flat connection, then $E$ admits a  $(q+l)$-contact structure. The construction, due to Almeida \cite{almeida_contact_2018}, is as follows. The connection induces a splitting $TE = H \oplus V$ into horizontal and vertical bundles. Flatness implies both bundles are involutive \cite[Lemma 3.1]{TondeurGeometryFoliations1997}. Let $\{X_1, \cdots, X_l\}$ be the commutative frame generating the $\mathbb{T}^l$-action on each fibre, so that $V = \mathrm{Span}\{X_1, \cdots X_l\}$. Define forms $\alpha_i$ on $M$ by setting 
    \begin{equation*}
        \begin{cases}
            \alpha_i(X_j) = \delta_{ij}, \\ \alpha_i\rvert_H = 0.
        \end{cases}
    \end{equation*}
    These are closed forms. Indeed, $[X_j, X_r] = 0$ for any choice of $j,r$. For any horizontal vector field $W$, the section $[X_j, W] $ is also horizontal \cite[Lemma 3.12]{nesterov_principal_2000}, hence
    \begin{align*}
        d \alpha_i(X_j, X_r) &= X_j\alpha_i(X_r) - X_r\alpha_i(X_j) - \alpha_i([X_j, X_r]) =0, \\
        d \alpha_i(X_j, W_1) &= X_j\alpha_i(W_1) - W_1\alpha_i(X_j) - \alpha_i([X_j, W_1]) =0.
    \end{align*}
    Since $H$ is involutive, given $W_1, W_2 \in H$,
    \begin{equation*} 
        d \alpha_i(W_1, W_2) = W_1\alpha_i(W_2) - W_2\alpha_i(W_1) - \alpha_i([W_1, W_2]) = 0,
    \end{equation*}
    hence $d\alpha_i\rvert_H \equiv 0$, and therefore each $\alpha_i$ is closed.
            
    \par Each fibre $H_p$ of $H$ is identifiable with the fibre $T_{\pi(p)}M = \mathcal{R}_{\pi(p)} \oplus \xi_{\pi(p)}.$ Thus $H$ has a splitting $H = \overline{\mathcal{R}} \oplus \widetilde{\xi}$, where $\overline{\mathcal{R}} \approx \mathcal{R}$ and $\xi \approx \widetilde{\xi}$. Write $\widetilde{\mathcal{R}} := V \oplus \overline{\mathcal{R}}$ so that we have 
    \begin{equation*}
        TE = \widetilde{\mathcal{R}} \oplus \widetilde{\xi}.
    \end{equation*}
    For any choice of indices $\{j_1, \cdots, j_l\} \subset \{1, \cdots, q\}^l$, we set     \begin{align*}
        \eta_{i} := \pi^\ast\lambda_{i}, ~~~~&\text{ for } 1 \leq i \leq q;\\
        \eta_{q+i} := \alpha_i + \pi^\ast\lambda_{j_i}, ~~~~&\text{ for } 1 \leq i \leq l.
    \end{align*}
    These forms are such that, for $i = 1, \cdots, q$
    \begin{equation*}
        \ker\eta_i = \pi^\ast\xi = \widetilde{\xi}\oplus \widetilde{\mathcal{R}};
    \end{equation*}
    and, for $i = 1, \cdots, l$
    \begin{equation*}
        \ker \widetilde{\eta}_{q+i} = \mathrm{Span}\{X_j\}_{j\neq i} \oplus \mathrm{Span}\{X - \lambda_{j_i}(\pi_\ast X)X_i; ~X \in \Gamma(H)\} \supset \widetilde{\xi}.
    \end{equation*}
    It then follows that
    \begin{align*}
        \bigcap_{i=1}^{q+l} \ker\eta_i 
        &= \left(\widetilde{\xi}\oplus \widetilde{\mathcal{R}}\right)\cap\left(\bigcap_{i=i}^l\mathrm{Span}\{X - \lambda_{j_i}(\pi_\ast X)X_i; ~X \in \Gamma(H)\}\right) \\ 
        &= \widetilde{\xi},
    \end{align*}
    and $(E, \vec{\eta}, \widetilde{\mathcal{R}} \oplus \widetilde{\xi})$ is therefore a  $(l+q)$-contact structure. Since the bundle $\overline{\mathcal{R}} \subset H$ is diffeomorphic to $\mathcal{R}$, we can think of the Reeb vector fields on $M$ as vector fields on $H$. Hence the Reeb vector fields of the $\eta_i$ are
    \begin{align*}
        \widetilde{R}_{i} := R_i, ~~~~&\text{ for } 1 \leq i \leq q;\\
        \widetilde{R}_{q+i} := \frac{1}{2}(X_i + R_{j_i}), ~~~~&\text{ for } 1 \leq i \leq l.
    \end{align*}
        
    Note that the restriction of the flat connection on $TE$ to $\widetilde{\mathcal{R}}$ is again flat, providing us with two complementary integrable sub-bundles $V$ and $\mathcal{R}$ of $\widetilde{\mathcal{R}}$. This means the leaves of $\widetilde{\mathcal{F}}$ are the product of the leaves of $\mathcal{F}$ with the torus $\mathbb{T}^l$.
\end{example}

In particular, by taking $M$ to be a closed contact manifold and choosing the dimension of the fibre appropriately, we can construct closed manifolds supporting $q$-contact structures for any $q$ larger than $1$. 

Finally, we have a classic example of algebraic nature from the theory of semi-simple Lie groups.

\begin{example}[Weyl Chamber Actions]
	Suppose $G$ is a semi-simple Lie group and $\mathfrak{a}$ is a Cartan subspace with centraliser $\mathfrak{a}\oplus \mathfrak{k}$. Let $K$ be a compact subgroup of $G$ associated with $\mathfrak{k}$ and $\Gamma$ a uniform lattice on $G$ acting freely on $G/K$. The exponentials of elements in $\mathfrak{a}$ act on $\Gamma\setminus(G/K)$ by right translation. This is called the \textbf{Weyl Chamber Action}. Almeida proved \cite[Theorem 1.0.2]{almeida_contact_2018} that this is a $q$-contact action of $\mathbb{R}^q$.

\end{example}

\subsection{Local Charts and Reparameterisations}

\par Now we begin to examine what the forms of an adapted coframe look like in local coordinates, providing properties to their coordinate functions. Later on, in subsection \ref{taut}, we will improve these results, showing that the defining forms can also be made into harmonic forms. This property will be essential when studying isometric contact actions in Section \ref{partres}. We will also discuss how one can change from one adapted coframe to another while preserving the Reeb and contact distributions, a process we call \textit{reparameterisation.}

\par In a foliated chart $\psi: U \to \mathbb{R}^{2n} \times \mathbb{R}^q,$ the tangent coordinates $(z_1, \cdots, z_q)$ can be chosen so that $\partial_{z_i} = R_i$ (cf \cite[Theorem 8.3]{BoothbyIntroductionDifferentiableManifolds1986}). Evaluating the forms $\lambda_i$  and their derivatives at these fields give rise to the following local characterisation for the adapted coframe $\vec{\lambda}$.
    
\begin{proposition}\label{lambdacoord}
    Around each point $p \in M$ there are coordinates $(x,y,z) \in \mathbb{R}^{2n}\times\mathbb{R}^q$, and, for each $1\leq i\leq q$ and $1\leq j\leq n$, functions $f^i_j, g^i_j: \mathbb{R}^{2n}\to\mathbb{R}$ such that:
    \begin{equation*}
        \lambda_i = d z_i + \sum_{j=1}^n\left( -f^i_j(x,y)d x_j + g^i_j(x,y) d y_j\right).
    \end{equation*}
\end{proposition}
    
\par Motivated by the last proposition, one can ask whether $q$-contact structures are locally isomorphic or not, in the sense of a generalised Darboux Theorem. A partial result in this direction comes from the following proposition.
\begin{proposition}\label{localRepresentationCodimension2}
    Let $M$ be a manifold supporting a contact foliation of codimension 2, that is, $\dim M = 2 + q$. Then around each point of $M$ there are coordinates $(U; x,y,z_1, \cdots, z_q)$ and leaf-wise constant functions $a_i: U \to \mathbb{R}$ such that $\partial_xa_i$ is non-vanishing and 
    \begin{equation*}
        \lambda_i = d z_i + a_id y.
    \end{equation*}
    Moreover, we can take $a_1 = x$.
\end{proposition}
\begin{proof}
    Around each point of $M$, one can find a neighbourhood $U$ diffeomorphic to a product of discs $D^q \times D^2$. Let $\Phi: U \to D^2$ be the submersion defined by this diffeomorphism, the fibres of which are the plaques of $\mathcal{F}$ on $U$. Since the symplectic form $d\lambda_i$ is holonomy invariant, there is a well-defined symplectic form $\omega_i$ on $D^2$ satisfying $\Phi^\ast\omega_i = d\lambda_i$. By Darboux's theorem, there is a coordinate chart $(x',y')$ on $D^2$ such that $\omega_1 = d x'\wedge d y'$. Moreover, for every other $i$ there is a non-vanishing function $b_i$ on $D^2$ such that $\omega_i = b_id x'\wedge d y'$. 
    \par Consider on $U$ the unique fields $X$ and $Y$ on $M$ that are $\Phi$-related to the coordinate vector fields $\partial_{x'}$ and $\partial_{y'}$ and have no components in any leaf-wise direction. More precisely, they satisfy:
    \begin{equation*}
    \begin{cases}
        X = \Phi^\ast \partial_{x'}\\
        Y = \Phi^\ast \partial_{y'}\\   
        \mathcal{L} _{R_i}X = \mathcal{L} _{R_i}Y = 0 \text{   for all } i.
    \end{cases}
    \end{equation*}
    Thus, $\{R_1, \cdots, R_q, X, Y\}$ is a local frame for the tangent bundle of $M$, satisfying $[R_i, R_j] = [X, R_j] = [Y, R_j] = 0$ for any $i,j$. As for $[X,Y],$ we have
    \begin{equation*}
        \lambda_i([X,Y]) =  X\lambda_i(Y) - Y\lambda_i(X) - d\lambda_i(X,Y) = b_id x'\wedge d y'(\partial_{x'},\partial_{y'}) = b_i \circ \Phi,
    \end{equation*}
    so that $[X,Y] = \sum_i(b_i \circ \Phi)R_i$. 
    We define on $D^2$ functions $c_i(x',y') = \int_0^{x'}b_i(t,y')d t$ and note that, since $X$ and $\partial_{x'}$ are $\Phi$-related, these satisfy 
    \begin{equation*}
        X(c_i\circ\Phi) = \partial_{x'}c_i \circ \Phi = b_i\circ\Phi.
    \end{equation*}
    Thus, the vector field $\widetilde Y :=  Y - \sum_i(c_i \circ \Phi)R_i$ is linearly independent from $X$ and the $R_i$; in addition, it is such that 
    \begin{equation*}
        [X,\widetilde Y] = [X,Y] - \sum_i[X,(c_i \circ \Phi)R_i] = \sum_i(b_i \circ \Phi)R_i -\sum_i(b_i \circ \Phi)R_i = 0.
    \end{equation*}
    One can, therefore, find coordinates $(x,y,z_1, \cdots, z_q)$ on $M$ whose coordinate vector fields are $\{X, \widetilde{Y}, R_1, \cdots, R_q\}$, obtaining
    \begin{equation*}
    \begin{cases}
        X = \partial_{x}\\
        Y = \partial_{y} + \sum_i(c_i\circ\Phi)\partial_{z_i} \\ 
        R_i = \partial_{z_i}.
    \end{cases}
    \end{equation*}
    It follows from construction that $\Phi^\ast d x' = d x$ and $\Phi^\ast d y' = d y$. Hence 
    \begin{equation*}
        d\lambda_i = \Phi^\ast\omega_i = (b_i\circ \Phi)d x \wedge d y
    \end{equation*}
    and, since $d (c_i\circ\Phi) = (b_i\circ\Phi)d x + \partial_y(c_i\circ\Phi)d y$,  
    \begin{equation*}
        d\left((c_i\circ\Phi)d y\right) =  (b_i\circ\Phi)d x \wedge d y.
    \end{equation*}
    Thus $d(\lambda_i - (c_i\circ\Phi)d y) = 0$, and there exists, on a possible smaller neighbourhood, a function $f_i$ such that 
    \begin{equation*}
        \lambda_i = d f_i - (c_i\circ\Phi)d y.
    \end{equation*}
    
    \noindent Finally, we see that
    \begin{equation*}
    \begin{split}
        \delta_{ij} &= \lambda_i(R_j) = df_i(\partial_{z_j}) - (c_i\circ\Phi)d y(\partial_{z_j}) = \partial_{z_j}f_i; \\
        0 &= \lambda_i(X) = d f_i(\partial_x) - (c_i\circ\Phi)d y(\partial_x) = \partial_xf_i; \\
        0 &= \lambda_i(Y) = d f_i(\partial_y) + \sum_j(c_j\circ\Phi)d f_i(\partial_{z_j}) - (c_i\circ\Phi)d y(\partial_y) - \sum_j(c_j\circ\Phi)d y(\partial_{z_j}) \\ 
        &= \partial_yf_i + (c_i\circ\Phi)-(c_i\circ\Phi) = \partial_yf_i,
    \end{split}
    \end{equation*}
    and therefore $d f_i = d z_i$, from where we conclude, setting $a_i :=  -(c_i\circ\Phi)$, that $\lambda_i = d z_i +a_id y$, as we wanted.
 
\end{proof}
    
A crucial fact in the construction above is the existence of the non-vanishing functions $b_i$. When the codimension is greater than $2$, one can not generally find a single basis in which all the symplectic forms are written canonically, so this argument does not hold. When, however, all the derivatives $d\lambda_i$ coincide, then a similar construction in general codimension yields:
\begin{proposition}[\cite{blair_darboux_2006}]\label{localRepresentationUniformCase}
    If $M$ is an uniform $q$-contact manifold, then around each point there are coordinates $(U; x,y,z_1, \cdots, z_q)$ such that 
    \begin{equation*}
        \lambda_i = d z_i + \sum_j x_jd y_j.
    \end{equation*}
\end{proposition}

\begin{lemma}[Reparameterisation Lemma]\label{replem} 
    Let $\vec{\lambda} = (\lambda_1, \cdots, \lambda_q)$ be an adapted coframe for a $q$-contact structure on $M$, and $A(p) = \{a_{ij}(p)\}$ a mapping $A: M \to \mathcal{M}(q; \mathbb{R})$ from $M$ to the space of real matrices. If $A$ is sufficient $C^1$-close to $0$, and the functions $a_{ij}$ are leaf-wise constants (with respect to the contact foliation), then 
    \begin{equation*}
        \vec{\eta} := (\mathrm{id} - A)\vec{\lambda}
    \end{equation*}
    is also an adapted coframe for a $q$-contact structure on $M$, and the splittings associated with these two structures are the same.
\end{lemma}
    
\begin{proof}
    We have 
    \begin{equation*}
        \eta_i = \lambda_i - \sum_ja_{ij}\lambda_j,
    \end{equation*}
    and therefore $\xi = \cap_j\ker\lambda_j \subset \ker\eta_i$, for every $i$. On the other hand, by choosing $A$ small enough so that $(\mathrm{id}-A)$ is invertible, we can write each $\lambda_i$ as linear combination of the $\eta_j $, thus obtaining that $\cap_j\ker\eta_j \subset \ker\lambda_i$, hence
    \begin{equation*}
        \cap_j\ker\eta_j = \cap_j\ker\lambda_j = \xi.
    \end{equation*} 
    Moreover, 
    \begin{equation}\label{diffeta}
         d \eta_i =  d \lambda_i -\sum_j( d  a_{ij}\wedge\lambda_j + a_{ij} d \lambda_j),
    \end{equation}
    and, since by hypothesis the functions $a_{ij}(R_l) = 0$ for every $i,j,l$, the Reeb vector fields $R_i$ satisfy $\iota_{R_i} d \eta_j = 0$. 
    In other words, $\mathcal{R} \subset \ker d \eta_i$, for every $i$. 
    The splitting $\mathcal{R} \oplus \xi = T M$ and the non-degeneracy of $d\eta_i$ on $\xi$ imply $\ker d \eta_i \subset \mathcal{R}$, and we have equality between these two bundles as well. 
    \par Using Equation \ref{diffeta}, it is easy to see that 
    \begin{equation*}
         d \eta_i^n =  d \lambda_i^n + \epsilon_i,
    \end{equation*}
    where $\epsilon_i$ can get arbitrarily small if $A$ is taken small enough. In particular, for $\epsilon_i$ sufficiently close to $0$, $d\eta_i^n$ is volume form on the bundle $\xi$. Moreover, by employing the Determinant Theorem, we can derive the equality
    \begin{equation*}
        \eta_1\wedge\cdots\wedge\eta_q\wedge( d \eta_i)^n = \det(\mathrm{id}-A)\lambda_1\wedge\cdots\lambda_q\wedge(( d \lambda_i)^n + \epsilon_i),
    \end{equation*}
    where the RHS will be a volume form as long as $\epsilon_i$ is sufficiently small.
    
\end{proof}
    
In general, the Reeb vector fields of the $\eta_i$ are not the same as the $R_i$ (though they spam the same invariant bundle), so the construction above might be thought of as a \textit{reparameterisation of the contact action}.

When the action is transitive, the situation is considerably better, as we do not need the smallness conditions and actually have a dense subset of possible reparameterisations from which to choose. This happens because the existence of a dense leaf implies that any leaf-wise constant function is just a constant function, so in the transitive case, the only reparameterisations are honest matrices in $\mathrm{Gl}_q( \mathbb{R})$. One can then show that a matrix only fails to take adapted coframes into adapted coframes if its entries are a zero of certain polynomials, that is, the set of such matrices form a Zariski closed subset of $\mathbb{R}^{q\times q}$. 
    
\begin{lemma}\label{densereplem}\cite{almeida_contact_2018}
    Suppose the contact foliation $\mathcal{F}$ is transitive, and let $\mathcal{B} \subset \mathrm{Gl}_q(\mathbb{R})$ be the set of (constant) reparameterisations of the action. The complement of $\mathcal{B}$ is closed and has empty interior. In particular, $\mathcal{B}$ is dense in $\mathrm{Gl}_q(\mathbb{R}).$
\end{lemma}

\subsection{Generalised Weinstein Conjectures}
    
We wish to investigate whether or not some Weinstein-like conjecture holds for the more general objects of Definition \ref{qcontactstructure}. With this in mind, we propose two generalisations for the Weinstein conjecture.
    
\begin{conjecture}[The Weak Generalised Weinstein Conjecture ($\mathrm{WGWC}$)] 
    A contact foliation on a closed manifold $M$ cannot be a foliation by planes.
\end{conjecture}

\begin{conjecture}[The Strong Generalised Weinstein Conjecture ($\mathrm{SGWC}$)]
    Every contact foliation has a closed leaf, that is, a leaf homeomorphic to a torus $\mathbb{T}^q$.
\end{conjecture}
    
We remark that being orbits of a locally free action of $\mathbb{R}^q$, every leaf of the foliation $\mathcal{F}$ is homeomorphic to $\mathbb{R}^{q-l} \times \mathbb{T}^l$ for some $0 \leq l \leq q$, which depends on the leaf. The $\mathrm{WGWC}$ states that $l \geq 1$ for some leaf, while the $\mathrm{SGWC}$ asks for the existence of a leaf for which $l = q$. Of course, in the contact case, that is, when $q = 1$, a leaf that is not homeomorphic to $\mathbb{R}$ is automatically a closed curve, homeomorphic to $S^1$. Therefore, both the $\mathrm{WGWC}$ and the $\mathrm{SGWC}$ are equivalent to the Weinstein Conjecture on dimension $1$ and comprise generalisations of this conjecture to higher dimensional contact foliations. They form a hierarchy
\begin{equation*}
    \mathrm{SGWC} \implies \mathrm{WGWC} \implies \mathrm{WC}, 
\end{equation*}
with the converse implications holding when $q = 1$.

\par Recall that a subset of a foliated space is called \emph{invariant} if it is closed and saturated. A foliation is \emph{minimal} if its ambient manifold contains no invariant subset. Equivalently, a foliation is minimal when every leaf is a dense subset in the ambient manifold. Since closed leaves are invariant, it is clear that if the $\mathrm{SGWC}$ holds, then there can be no minimal contact foliations. A priori, there is nothing to prevent the existence of minimal contact foliations satisfying the $\mathrm{WGWC}$. We wish to show that this scenario also can not happen. This is a corollary to the following ``reduction'' procedure, which can be seen as a partial converse to the ``extension'' construction of Example \ref{flatbundleexample}. First, let us recall that the relation 
\begin{equation*}
    \mathcal{F}(x) := ~^{\textstyle \mathbb{R}^q}\!\big/_{\textstyle \mathrm{Iso}(\mathcal{F}(x))},
\end{equation*}
where $\mathrm{Iso}(\mathcal{F}(x)) := \{a \in \mathbb{R}^q; F(a,x) = x\}$ is the isotropy group of the leaf (i.e., \emph{orbit}) $\mathcal{F}(x)$, determines the topological type of a leaf  . The kernel of the action is a lattice in $\mathbb{R}^q$, consisting of all the elements acting as the identity on $M$, and it equals the  intersection of all the isotropy subgroups:
\begin{equation*}
    \bigcap_{x \in M}\mathrm{Iso}(\mathcal{F}(x)) = \mathrm{ker}F := \{a \in \mathbb{R}^q; F(a, \cdot) = \mathrm{id}\} \approx \mathbb{Z}^l,
\end{equation*}
hence every leaf is a cylinder $\mathbb{T}^s\times\mathbb{R}^{q-s}$ for some $s \geq l$. In particular, if the action's kernel is non-trivial, then no leaf is a plane. An action whose kernel is trivial is called \emph{faithful}.

\begin{proposition}[Reduction of a contact action]\label{contredux}
    Let $(M, \vec{\lambda}, \mathcal{R} \oplus \xi)$ be a $q$-contact manifold. Suppose the contact action $F$ is transitive and non-faithful. Then $M$ is a principal $\mathbb{T}^l$-bundle over a closed $(2n+q-l)$-dimensional manifold $M_0$, for a fixed $0 < l <q$; moreover, $F$ induces a faithful $(q-l)$-contact action on $M_0$. 
\end{proposition}
\begin{proof}
    Let $\Gamma := \ker F \approx \mathbb{Z}^l$, where $0 < l < q$. We consider the vector spaces
    \begin{align*}
        G &:= \mathrm{Span}~\Gamma \approx \mathbb{R}^l, \\
        H &:= ~^{\textstyle \mathbb{R}^q}\!\big/_{\textstyle G} \approx \mathbb{R}^{q-l},
    \end{align*}
    and the natural isomorphism $\mathbb{R}^q \approx G \oplus H.$  
    \par First, let the torus $\mathbb{T}^l \approx \Gamma / G$ act on $M$ by 
    \begin{equation*}
        (a + \Gamma)\cdot x := F(a,x).
    \end{equation*}
    This action is well defined because if $a + \Gamma = b + \Gamma$, their difference belongs to the kernel of $F$. It is also a free action, since $(a + \Gamma)\cdot x = x$ means $F(a, \cdot)$ acts as the identity on the leaf $\mathcal{F}(x)$, hence $a \in \Gamma,$ the isotropy group of $\mathcal{F}(x)$. Finally, the action is proper due to the compactness of  $\mathbb{T}^l$. It follows that the leaf space
    \begin{equation*}
        M_0 := ~^{\textstyle M}\!\big/_{\textstyle \mathbb{T}^l}
    \end{equation*}
    is a $(2n+q-l)$-dimensional closed manifold $M_0$, and $\mathbb{T}^l \hookrightarrow M \xrightarrow[]{\rho} M_0$ is a principal $\mathbb{T}^l$-bundle, where $\rho: M \to M_0$ is the canonical projection.
    \par We define an action of $\mathbb{R}^{q-l}$ on $M_0$, via $H$, by
    \begin{equation}\label{reduxact}
    \begin{split}
        F_0: H \times M_0 &\rightarrow M_0 \\
        (\overline{a}, \rho(x)) &\mapsto \rho(F(a,x)).
    \end{split}
    \end{equation}
    This does not depend on the first representative $a$, since $\rho(F(a,x)) = \rho(F(b,x))$ for every $a, b \in \mathbb{R}^q$. It also does not depend on the second representative $x$. Indeed, if $\rho(x)= \rho(y)$, then $y = F(b, x)$ for some $b \in G$, and consequently $F(a, y) = F(a+b, x)$ belong to the same leaf as $F(a,x)$.
    \par It remains to show that $F_0$ is a contact action. We will achieve this by using integration along the fibres, possibly after choosing a suitable reparameterisation of the action $F$. Since the bundle $\mathcal{R}$ is trivial, the isomorphism $\mathbb{R}^q \approx G \oplus H$ induces a splitting $\mathcal{R} = \mathcal{G} \oplus \mathcal{H}.$ Note that $\mathcal{G}_x$ is composed of the directions tangent to the fibre $\rho^{-1}(\rho(x))$, hence $\mathcal{G}_x = \ker d\rho_x$. Similarly, the tangent space at $\rho(x)$ of the orbit $F_0(H, \rho(x))$ is exactly $\rho_\ast(\mathcal{H}_x)$, so it is sufficient to show that $\mathcal{H}_0 := \rho_\ast\mathcal{H}$ can be realised as the Reeb bundle of a $(q-l)$-contact structure on $M_0$.
    \par We begin by decomposing the Reeb fields of $\vec{\lambda}$ as
    \begin{equation*}
        R_i = R_i^\mathcal{G} \oplus R_i^\mathcal{H}.
    \end{equation*}
    Using Lemma \ref{replem} to find a suitable reparameterisation if necessary (recall $\mathcal{F}$ is a transitive foliation), we may assume without loss of generality that 
    \begin{equation*}
    \begin{split}
        \mathcal{H} &= \mathrm{Span}\{R_1^\mathcal{H}, \cdots, R_{k-l}^\mathcal{H}\}; \\
        \lambda_i (R_i^\mathcal{H}) &\neq 0, \text{for } i= 1, \cdots, k-l.
    \end{split}    
    \end{equation*}
    Let us further replace $R_i^\mathcal{H}$ by a suitable multiple $X_i$ such that $\lambda_i(X_i) \equiv 1$. In addition, we consider a fibre-wise volume form $\omega$, normalised as to satisfy
    \begin{equation*}
        \int\displaylimits_{\rho^{-1}(y)}\!\omega = 1,
    \end{equation*}
    for every $y \in M_0$. We define $k-l$ differential $1$-forms on $M_0$ by $\eta_i := \rho_\ast(\lambda_i \wedge\omega)$, using the linear morphism $\rho_\ast: \wedge^\ast(M) \to \wedge^{\ast-l}(M_0)$, determined by $\rho$ via integration along the fibres. To be more precise, 
    \begin{equation*}
        \eta_i\rvert_y(Z) = \int\displaylimits_{\rho^{-1}(y)}\!\iota_{\widetilde{Z}}(\lambda_i\wedge\omega),
    \end{equation*}
    where the RHS is independent of the choice of lifting $\widetilde{Z}$. The forms $\eta_i$ are non-vanishing since $\eta_i(\rho_\ast X_i) \equiv 1$ for $i = i, \cdots q-l$; it is clear from their construction that they are linearly independent forms whose restriction to $\xi_0:= \rho_\ast\xi$ is identically zero. Finally, the morphism $\rho_\ast$ commutes with the exterior derivative, from where it follows that each $d\eta_i$ is non-degenerate on $\xi_0$, and has as its kernel the bundle $\rho_\ast\mathcal{R}=\rho_\ast\mathcal{H} =: \mathcal{H}_0$. Hence $(M_0, \vec{\eta}, \mathcal{H}_0\oplus \xi_0)$ is a $(q-l)$-contact manifold, as we wanted.
    
\end{proof}
  
\begin{theorem}\label{equivconjmin}
  A \emph{minimal} contact foliation $(M, \mathcal{F})$ is either a minimal foliation by planes, or it reduces to a minimal contact foliation by planes.
\end{theorem}
\begin{proof}
	Since $(M, \mathcal{F})$ is a minimal contact foliation, all leaves are dense in $M$, and their isotropy groups are all the same, namely the lattice $\Gamma:= \ker F \approx \mathbb{Z}^l$. None of these leaves can be closed, hence $l < q$. If $l=0$, then $M$ is a foliation by planes. Otherwise, it follows from Proposition \ref{contredux} that $\rho: M \to M_0$ is a principal bundle over a $(q-l)$-contact manifold $M_0$. We claim that the contact action defined by  (\ref{reduxact}) is a minimal action whose orbits are all planes. Indeed, the action is minimal: given $\rho(x), \rho(y) \in M_0$, the leaf $\mathcal{F}(x)$ accumulates on $y$, hence there is a sequence $a_n \in \mathbb{R}^q$ such that 
    \begin{equation*}
        F(a_n, x) \to y.
    \end{equation*}
    We write $a_n = a_n^G + a_n^H$, and note that $F(a_n^G, \cdot)$ gets arbitrarily close to the identity. We choose a sub-sequence $a_{n_j}$ such that $F(a_{n_j}^G, x) \to x$, from where we conclude $F(a_{n_j}^H, F(a_{n_j}^G, x)) \to y$, implying
    \begin{equation*}
        F_0(\overline{a_{n_j}}, \rho(x)) = F_0(\overline{a_{n_j}}, \rho(F(a_{n_j}^G, x))) = \rho(F(a_{n_j}^H, F(a_{n_j}^G, x))) \to \rho(y).
    \end{equation*}
    Therefore, $\mathcal{F}_0(\rho(x))$ accumulates on $\rho(y)$, and every orbit of the action $F_0$ is dense. To see now that the orbits are planes $\mathbb{R}^q,$ it is simply a matter of noticing that $F_0$ is free. If $\rho(x) = F_0(\overline{a}, \rho(x)) := \rho(F(a,x))$, then there is $b \in G$ such that 
    \begin{align*}
        F(a,x) = (b+\Gamma)\cdot x = F(b, x).
    \end{align*}
    Thus $a$ acts like an element of $G$, representing the identity on $H$. In particular, minimality implies that the isotropy group of any leaf is $\ker F_0 = \{0\}$. Hence every leaf is a plane.
    
\end{proof}

\begin{corollary}
If the $\mathrm{WGWC}$ is true, then there can be no minimal contact foliations.
\end{corollary}

\subsection{Characteristic Foliations and transversals}

\par In this subsection, we look briefly into the holonomy of a contact foliation. Moreover, we show that each $1$-form $\lambda_i$ can be restricted to a contact form on suitable submanifolds of $M$, which will be given as transversals to a characteristic foliation $\mathcal{F}_i$ associated to $\lambda_i$. These first simple properties can be summed up in the following Proposition.

\begin{proposition}\label{holonomyproperties}
    Let $M$ be a $q$-contact manifold.
    \begin{itemize}
        \item[(i)] The holonomy pseudogroup $\mathcal{H}$ of the foliation $\mathcal{F}$ tangent to the orbits of the action consists of symplectomorphisms.
        \item[(ii)] The holonomy pseudogroup $\mathcal{H}_i$ of each characteristic foliation $\mathcal{F}_i$ consists of contactomorphisms. 
    \end{itemize}
    In particular, all the holonomy maps above are volume-preserving transformations.
\end{proposition}

\par We begin by constructing the foliations $\mathcal{F}_i$.
    
\begin{definition}[\textit{Characteristic distribution}]
 	The \textbf{characteristic distribution} of $\lambda_i$ is the distribution 
	\begin{align*}
        		\mathcal{C}_i(p) &:= \{X_p \in T_pM; ~ \lambda_i(X_p) = 0 \text{  and  } \iota_{X_p} d \lambda_i = 0\} \\
       		 &= \ker\lambda_i\rvert_p \cap \mathcal{R}_p \\
       		 &=\mathrm{Span}\{R_1\rvert_p, \cdots, \widehat{R_i\rvert_p}, \cdots, R_q\rvert_p\}.
   	 \end{align*}
\end{definition}
    
\par Note that this distribution has a constant rank equal to $q-1$ (in particular, they are trivial for contact manifolds). More than that, if the splitting $\mathcal{R} \oplus \xi$ is $C^k$, so are the characteristic distributions $\mathcal{C}_i$. Because the Reeb fields are pairwise commutative, the characteristic distributions are all integrable. In fact, the characteristic foliations are the orbit foliations of the $\mathbb{R}^{q-1}$ ``sub-actions'' $F_i: \mathbb{R}^{q-1} \times M \rightarrow M$ given by
\[
        	(t_1, \cdots, t_{q-1}, p) \mapsto (\phi_1^{t_1}\circ \cdots \circ \phi_{i-1}^{t_{i-1}} \circ \phi_{i+1}^{t_i}\circ \cdots\circ \phi_q^{t_{q-1}})(p).
\]
The underlying foliation $\mathcal{F}_i$ associated with the characteristic distribution $\mathcal{C}_i$ is called the \textbf{characteristic foliation of $\lambda_i$}.

\par Due to the nature of the $q$-contact action, the foliation $\mathcal{F}$ admits no closed transversal.
\begin{proposition}
    There is no $2n$-dimensional closed submanifold of $M$ everywhere transverse to $\mathcal{F}$.
\end{proposition}
\begin{proof}
    Indeed, if $N$ was one such manifold, then the exact $2n$-form
    \begin{equation*}
        \iota_{R_1}(\iota_{R_2} \cdots \iota_{R_q}(d  M_i))\cdots) = (d \lambda_i)^n = d (\lambda_i\wedge (d \lambda_i)^{n-1})
    \end{equation*}
    would be an exact volume form on $N$, contradicting Stokes's theorem.
    
\end{proof}
In particular, there are no complete transversals, so no contact foliation is equivalent to a suspension. On the other hand, in general, the characteristic foliations admit transverse submanifolds, and each of these transversals is a contact manifold.
\begin{proposition}\label{transversalsarecontact}   
    If $N$ is a transversal of $\mathcal{F}_i$, then $\lambda_i\rvert_N$ is a contact form on $N$.
\end{proposition}
\begin{proof}
    The equality $\mathcal{C}_i = \mathrm{Span}\{R_j\}_{j\neq i}$ implies that each leaf of $\mathcal{F}_i$ is transverse to both $R_i$ and $\xi$. Suppose that $N$ is a transversal of $\mathcal{F}_i$. Then $\dim N = 2n + 1$ and, as before, the fact that $N$ is transverse to $R_j$ for every $j$ other than $i$ means 
    \begin{equation*}
        \iota_{R_1}(\iota_{R_2}( \cdots \widehat{\iota_{R_i}}( \cdots \iota_{R_q}(d  M_i))\cdots) = \lambda_i\wedge d \lambda_i^n
    \end{equation*} 
    is a volume form on $N$. Hence $\lambda_i\rvert_N$ is a contact form.
    
\end{proof}
\begin{remark}\label{transversalsinduceReeborbits} 
    The contact structure associated is the intersection bundle 
    \begin{equation*}
        \ker\lambda_i \cap TN = \left(T\mathcal{F}_i \oplus \xi\right)\cap TN = \xi \cap TN.
    \end{equation*}
    We emphasise that by $\lambda_i\rvert_N$ we mean $j^\ast\lambda_i$, where $j: N \to M$ is an embedding. 
    In particular, $R_i$ need not coincide with the Reeb vector field of $\lambda_i\rvert_N$, \textit{though they share their closed orbits}. 
    Indeed, there is a bundle isomorphism $\Psi: \mathrm{Span}\{R_i\} \oplus \xi \to TN $ determined by fibre-wise projection. 
    In particular, the image $\Psi R_i$ is the unique vector field on $N$ such that $X_i := \Psi R_i - R_i \in TN$. Therefore
    \begin{equation*}
        \lambda_i(\Psi R_i) = \lambda_i(R_i) + \lambda_i(X_i) = 1
    \end{equation*}
    and 
    \begin{equation*} 
        d \lambda_i(\Psi R_i, \cdot) = d \lambda_i(R_i, \cdot) + d \lambda_i(X_i, \cdot) = 0,
    \end{equation*}
    so that $\Psi R_i$ is the Reeb vector field of $\lambda_i\rvert_N$. 
    Moreover, for any closed orbit $\gamma: \mathbb{R} \to N$ of $\Psi R_i$ of period $T$, we have 
    \begin{equation*}
        R_i(\gamma(s)) + X_i(\gamma(s)) = \Psi R_i(\gamma(s)) = \Psi R_i(\gamma(T+s)) = R_i(\gamma(T + s)) + X_i(\gamma(T + s))
    \end{equation*}
    and therefore 
    \begin{equation*}
        R_i(\gamma(s)) - R_i(\gamma(T + s)) = X_i(\gamma(T + s)) -X_i(\gamma(s)),
    \end{equation*}
    where the LHS belongs to $\mathrm{Span}{R_i}\oplus\xi$ while the RHS belongs to $T\mathcal{F}_i$. Since these subspaces are complementary, it follows that $\gamma$ is a closed orbit of period $T$ for both $R_i$ and $X_i$. In particular, $R_i$ admits a closed orbit in $M$.
\end{remark} 

\par The holonomy maps of $\mathcal{F}_i$ connect different transversal by following the leaves of $\mathcal{F}_i$, effectively ``flowing'' along the vector fields $R_j$, $j\neq i$. As all these fields preserve the $q$-contact distribution, so do the holonomy maps, implying: 
\begin{proposition}
    Let $p \in M$ and $q\in \mathcal{F}_i(p)$, and $N_p, N_q$ be transversals through $p$ and $q$, respectively. Then there are small transverse neighbourhoods $W_p$ and $W_q$ around $p$ and $q$ such that the holonomy map $h: W_p \to W_q$ is a contactomorphism. 
\end{proposition}
\par Using these same arguments, we see that for any transversal $N$ to the foliation $\mathcal{F}$, the restriction of $d \lambda_i$ to $N$ is a volume form. Again, flowing along the leaves of the foliation $\mathcal{F}$ also preserves the bundle $\xi$ and the $2$-forms $d \lambda_i$, so that we have 
\begin{proposition}
    The holonomy maps of a contact foliation are symplectomorphisms. 
\end{proposition}
\noindent The last results are exactly the content of Proposition \ref{holonomyproperties}.

\par The fact that transversals of characteristic foliations are contact manifolds can be used to relate the Weinstein Conjecture with its generalised versions. For instance, in the uniform case, one can construct closed transversals using techniques first developed by Tischler in \cite{tischler_fibering_1970}.
\begin{proposition}\label{propositionequalderivative} 
    Let $(M, \mathcal{F})$ be a closed $q$-contact manifold with adapted coframe $\vec{\lambda}$. If there is $i, j$ such that 
    \begin{equation*}
        d\lambda_i = d\lambda_j,
    \end{equation*}
    then $M$ is a fibre bundle over $S^1$. Furthermore, if, for any choice of indices ${i_1, \cdots, i_l} \subset \{1, \cdots, q\}$, one has $d\lambda_{i_1} = \cdots = d\lambda_{i_l}$, then $\dim H^1(M; \mathbb{R}) \geq l-1$.
\end{proposition}
\begin{proof}
    The $1$-form $\alpha := \lambda_i-\lambda_j$ is closed and nowhere vanishing. Hence Tischler's theorem \cite{tischler_fibering_1970} proves the first assertion. For the second one, note that the forms $\alpha_j:= \lambda_{i_1}-\lambda_{i_j}$ comprise a l.i. collection of closed forms representing non-trivial cohomology classes. 
    Indeed, the forms $\alpha_j$ can not be exact, since any function $f: M \to \mathbb{R}$ has a critical point $x$, which would imply $d f_x = \alpha_j\rvert_x \equiv 0$, which can not happen as $\alpha_j(R_{i_1}) \equiv 1$.
    
\end{proof}
    
\begin{corollary}
    A contact foliation defined by $\vec{\lambda} = (\lambda_1, \cdots, \lambda_q)$ on a sphere $S^{2n+q}$  can not have $d\lambda_{i_1} = d\lambda_{i_2}$ for more than one pair of indices $i_1, i_2$. In particular, for $q > 2$, spheres admit no uniform $q$-contact structure. 
\end{corollary}
    
We can improve Proposition \ref{propositionequalderivative} in the case of uniform $q$-contact structures, obtaining $M$ as a fibration over the torus $\mathbb{T}^{q-1}$.
    
\begin{theorem}\label{uniformStructuresFibrations} 
    Let $(M, \vec{\lambda}, \mathcal{R} \oplus \xi)$ be a closed uniform $q$-contact manifold, $q\geq 2$. Then $M$ fibres over $\mathbb{T}^{q-1}$. More specifically, for each characteristic foliation $\mathcal{F}_i$ there is a locally trivial fibration $\rho_i: M \to \mathbb{T}^{q-1}$ such that $(M, \mathcal{F}_i, \rho_i)$ is a foliated bundle. In other words, the foliation $\mathcal{F}_i$ is the suspension of an action of $\mathbb{Z}^{q-1}$ on a closed contact manifold $N_0$ whose dimension is $2n+1$.
\end{theorem}
\begin{remark} 
	We agree that our manifolds be at least of class $C^1$, and so are the fibrations $\rho_i$. If $M$ is $C^2$ or higher, $\rho_i$ can be taken to be at least $C^2$. 
\end{remark}
\begin{proof}
    We follow the steps of \cite[Corollary II]{tischler_fibering_1970}. For simplicity, let us work with the characteristic foliation $\mathcal{F}_q$, associated with $\lambda_q$, noting that this choice implies no loss of generality. We define 
    \begin{equation*} 
        \alpha_i :=  \lambda_i - \lambda_q, ~~i = 1, \cdots, q-1. 
    \end{equation*}
    As in Proposition \ref{propositionequalderivative}, these are linearly independent, non-vanishing closed forms. Moreover, $\alpha_i$ is a leaf-wise volume form for the foliation $\mathcal{L} _i$ defined by the flow of $R_i$, since $\alpha_i(R_i) \equiv 1$. In other words, the closed form $\alpha_i$ is \textit{transverse} to $\mathcal{L} _i$ (cf. \cite{sullivan_cycles_1976}).
    
    First, let us show that one can perturb the $1$-forms $\alpha_i$ without losing linear independence. Given a finite atlas for $M$, we can define suitable $C^r$ norms on $\wedge^1(M)$ as the maximum over all charts of the sum of $C^r$ norms of coordinate functions. 
    We consider on $\wedge^1(M)$ the topology obtained as direct limit of the $C^r$ norms, and denote by $\mathcal{D}_1$ the corresponding topological $\mathbb{R}$-vector space (cf.\cite{de_rham_varietes_1973}, \cite{CandelFoliationsI2000}). Due to the linear independence of $\{\alpha_i\}$, we can use the Hann-Banach theorem to find for each $\alpha_i$ a corresponding continuous dual (a \textit{current}) $f_i: \mathcal{D}_1 \to \mathbb{R}$ in the strong dual space $\mathcal{D}_1^\ast$, satisfying
    \begin{equation*}
        f_i(\alpha_j) = \delta_{ij}.
    \end{equation*}
    We define a mapping from the $(q-1)$-fold product of $\mathcal{D}_1$ with itself to $\mathbb{R}$ by
        \begin{align*}
            \Phi: \mathcal{D}_1 \times \cdots \times \mathcal{D}_1 &\rightarrow \mathbb{R} \\
            (\omega_1, \cdots, \omega_{q-1}) &\mapsto \det\{f_i(\omega_j)\}.
        \end{align*}
    This mapping is continuous, and since $\Phi(\alpha_1, \cdots, \alpha_{q-1}) = 1$, we can find an open neighbourhood $U$ of $(\alpha_1, \cdots, \alpha_{q-1})$ on which $\Phi$ is always positive. 
    This means the matrix $\{f_i(\omega_j)\}$ is invertible for any choice of $(\omega_1, \cdots, \omega_{q-1}) \in U$, and consequently the set  $\{\omega_1, \cdots, \omega_{q-1}\}$ is linearly independent for every element of $U$.
    
    Now, each $\alpha_i$ can be arbitrarily well approximated by a $C^2$ closed $1$-form $\omega_i$ with the properties:
    \begin{itemize}
        \item[(i)]the foliation $\mathcal{N}_i$ integral to the bundle $\ker \omega_i$ comprises the fibres of a $C^2$ fibration $\rho_i: M \to S^1$, that is, $\mathcal{N}_i(p) = \rho_i^{-1}(\rho(p))$;
        \item[(ii)] the leaves of $\mathcal{N}_i$ are transverse to $\mathcal{L} _i$.
    \end{itemize}
    See, for instance, \cite[Theorem 1]{tischler_fibering_1970} and \cite[Theorem 9.4.2 and Proposition 10.3.14]{CandelFoliationsI2000}) for better exposition of these results. Here we limit ourselves to emphasise that our contact foliations are always of class at least $C^1$ so that all such theorems can be applied to our setting. 
    
    Since the linear independence of a set of $(q-1)$ $1$-forms is an open property, we can choose the $\omega_i$ to be linearly independent. We set
    \begin{align*}
        \rho: M &\rightarrow \mathbb{T}^{q-1} \approx S^1\times\cdots\times S^1\\
        x &\mapsto (\rho_1(x), \cdots, \rho_{q-1}(x)).
    \end{align*}
    The mapping $\rho$ is a $C^2$ submersion since the $\omega_i$ are linearly independent; moreover, it is proper, being a map between compact manifolds. Erehsmann Fibration Theorem shows that $\rho$ is a (locally trivial) fibration. Each fibre $\mathcal{N}(p)$ of this fibration is an intersection
    \begin{equation*}
        \mathcal{N}(p) :=  \bigcap_{i=1}^{q-1}\mathcal{N}_i(p).
    \end{equation*}

    By construction, the fibres are transverse to each of the Reeb fields $R_1, \cdots, R_{q-1}$, and consequently to the characteristic foliation $\mathcal{F}_q$. Since the fibres are compact, it follows that the restriction of $\rho$ to each leaf of $\mathcal{F}_q$ is a covering map of the torus, and that $\mathcal{F}_i$ is the suspension of an action of $\pi_1(\mathbb{T}^{q-1}) \approx \mathbb{Z}^{q-1}$ on the typical fibre $N_0$ (cf. \cite[Chapter V]{camacho_geometric_2013}). Moreover, since $M$ is boundaryless, so is the fibre $N_0$. As in Proposition \ref{transversalsarecontact}, the restriction of $\lambda_q$ to $N_0$ induces a contact structure on the closed $(2n+1)$-manifold $N_0$. 
    
\end{proof}

\begin{corollary}\label{existtrans}
    If $(M, \vec{\lambda}, \mathcal{R} \oplus \xi)$ is a closed $C^1$ uniform $q$-contact manifold and $q\geq 2$, then every characteristic foliation $\mathcal{F}_i$ admits a complete closed transversal.
\end{corollary}

\begin{remark}
    Theorem \ref{uniformStructuresFibrations} is somewhat similar in taste to the characterisation of metric $f$-$K$-structures provided by Goertsches and Loiudice in \cite[Theorem 4.4]{goertsches_how_2020}. Note that a metric $f$-$K$-structures induces a uniform $q$-contact structure for which all the Reeb fields are Killing for some metric. We will deal with such structures in more detail in Section \ref{secriem}.
\end{remark}

In light of Proposition \ref{transversalsarecontact}, we have the following equivalence of Weinstein Conjectures:
    
\begin{theorem}\label{equivconjuniform}
    For closed uniform $q$-contact manifolds, the Weak Generalised Weinstein Conjecture is equivalent to the Weinstein Conjecture.
\end{theorem}
\begin{proof}
    The implication $\mathrm{WGWC} \implies \mathrm{WC}$ is evident since a flow line that is not homeomorphic to $\mathbb{R}$ has to be closed. Conversely, if the Weinstein Conjecture is true and $T_i$ is a closed transversal to $\mathcal{F}_i$, then $(T_i, \lambda_i\rvert_{T_i})$ has a closed Reeb orbit, being a closed contact manifold (cf. Proposition \ref{transversalsarecontact}). This, in turn, implies the existence of a closed orbit $\gamma$ for $\mathcal{L} _i$, as pointed out in Remark \ref{transversalsinduceReeborbits}. The leaf of $\mathcal{F}$ containing the closed orbit $\gamma$ is not a plane. Hence $M$ satisfy the $\mathrm{WGWC}$. 
    
\end{proof}
Since the $\mathrm{WC}$ holds in dimension $3$ \cite{taubes_seibergwitten_2007}, we have
\begin{corollary}
   Every Reeb field of a codimension $2$ uniform contact foliation on a closed manifold has a closed orbit. In particular, every such contact foliation satisfies the $\mathrm{WGWC}$.
\end{corollary}

Another consequence of Corollary \ref{existtrans} is the following, a direct application of a famous result of Plante \cite[Theorem 1.1]{plante_foliations_1975}.

\begin{corollary}
    If a closed manifold $M$ admits a uniform $q$-contact structure, then $H_{2n}(M; \mathbb{R})$ is non-zero.
\end{corollary}

Recall that an overtwisted contact $3$-manifold is one which contains an embedded overtwisted disk. In higher dimensions, the role of the overtwisted disk can be played by a \emph{Plastikstufe} \cite{niederkruger_plastikstufe_2006}, or, more generally, by overtwisted $2n$-closed balls \cite{borman_existence_2015}. In any case, it is known that such contact manifolds satisfy the $\mathrm{WC}$ \cite{albers_weinstein_2009,borman_existence_2015}. If we define a $q$-contact structure to be overtwisted when it ``contains'' an overtwisted contact structure in the classical sense, we can conclude that all such structures satisfy the $\mathrm{WGWC}$. More precisely:

\begin{definition}[\textit{Overtwisted $q$-contact structures}]
    A $q$-contact structure is \textbf{overtwisted} when, for any of the defining $1$-forms $\lambda_i$, the characteristic foliation $\mathcal{F}_i$ admits a transversal $T$ such that $(T, \lambda_i)$ is an overtwisted contact manifold.
\end{definition}

From Corollary \ref{existtrans} and Theorem \ref{equivconjuniform} it follows:

\begin{theorem}
    Every closed uniform overtwisted $q$-contact manifold satisfies the $\mathrm{WGWC}$.
\end{theorem}

\subsection{Tautness}\label{taut}    
This section shows that contact foliations satisfy a particular rigidity property. Namely, they are taut in the sense of Sullivan. We recall the concept of tautness for foliations of arbitrary codimension.
\begin{definition}[\textit{Harmonic foliations}] 
    A foliation is \textbf{geometrically taut} (also called \textbf{harmonic}) if there is a Riemannian metric on the ambient manifold for which every leaf is a minimal submanifold.
\end{definition}
It is a simple calculation to check that the characteristic form $\lambda$ satisfies
\begin{equation*}
    d\lambda(X_1, \cdots, X_{q+1}) \equiv 0
\end{equation*}
whenever at least $q$ of the fields $X_1, \cdots, X_{q+1}$ are tangent to $\mathcal{R}.$
In other words, $\lambda$ is \textit{$\mathcal{F}$-closed.} Hence, by Rummler's criterion \cite{rummler_quelques_1979}, $\mathcal{F}$ is taut, and the $2$-tensor on $\mathcal{R}$
    \begin{equation*}
        g^\tau = \sum_{i=1}^q\lambda_i\otimes\lambda_i
    \end{equation*}
is a leaf-wise Riemannian metric with respect to which every leaf of $\mathcal{F}$ is a minimal submanifold (has $0$ mean normal curvature). Adapting the calculations from \cite[Lemma 10.5.6]{CandelFoliationsI2000} we obtain:

\begin{proposition}\label{Fisharmonic}
    There is a Riemannian metric on $M$ such that any submanifold of dimension $2 \leq d \leq q$ is realisable as an orbit of the action of $\mathbb{R}^d$ induced by Reeb vector fields $R_{i_1}, \cdots, R_{i_d}$ is a minimal submanifold.
\end{proposition}

The metric, in this case, consists of the sum of $g^\tau$ with any metric $g^\perp$ on $\xi$. Since minimal curves are geodesics, the case $d=1$ says that the integral curves of $R_i$, when parameterised by arclength, are geodesics of $g$.

\begin{corollary}\label{flowlinesaregeodesics}
    There is a metric on $M$ such that each Reeb vector field is of unit length, their flow lines are geodesics, and the foliations $\mathcal{F}$ and $\mathcal{F}_i$ are harmonic foliations.
\end{corollary}

There are several links between harmonic foliations, harmonic functions, and harmonic differential forms (cf. \cite{kamber_harmonic_1982}). We are particularly interested in the following improved version of Proposition \ref{lambdacoord}.

\begin{proposition}\label{lambdacoordharmonic}
    Around each point $p \in M$ one can find coordinates $(x,y, z) \in \mathbb{R}^{2n}\times\mathbb{R}^q$ and, for each $1\leq i\leq q$ and $1\leq j\leq n$, functions $f^i_j, g^i_j: \mathbb{R}^{2n}\to\mathbb{R}$ such that:
        \begin{itemize}
            \item[(i)] $\lambda_i = d z_i + \sum_j\left(-f^i_j(x,y)d x_j + g^i_j(x,y) d y_j\right) \text{ for every } i$;
            \item[(ii)] each coordinate function $z_i$ and differential form $d z_i$ is harmonic.
        \end{itemize}
\end{proposition}
\begin{proof}
    Item (i) is just Proposition \ref{lambdacoord}.
    Given $p \in (M, g)$, let $U$ be an open geodesic ball of radius $\epsilon$ around $p$, and choose a foliated chart $\psi: U \to \mathbb{R}^{2n}\times\mathbb{R}^q$ such that $R_i = \partial_{z_i}$ and $\psi(p) = 0$. Then the inclusion
    \begin{equation*}
        t \mapsto (0, \cdots, 0, z_j(t), 0, \cdots, 0)
    \end{equation*}
    is an arclength parameterisation of the integral of $R_j$ through the point $p$, that is, a geodesic of $g$. In particular, the function $z_j$ is a Riemannian immersion. Its image is a minimal submanifold, hence $z_j$ is harmonic (cf. \cite[Paragraph $\S 2$]{eells_harmonic_1964}), and therefore so is its derivative $d z_j$ (see also \cite[Paragraph $\S 3$]{eells_harmonic_1964}). 
    This proves (ii). 

\end{proof}

\section{Partial results on the existence of compact leaves}\label{partres}
\subsection{Anosov Contact Foliations}
We briefly recall what it means for an action of $\mathbb{R}^q$ to be Anosov.
    
\begin{definition}[\textit{Contact Elements and Anosov Actions}] 
    Consider an $C^2$ action $\phi: \mathbb{R}^q \times M \to M$ of $\mathbb{R}^q$ on a closed Riemannian manifold $(M, g)$. A point $a \in \mathbb{R}^q$ is said to be an \textbf{Anosov element of the action} if $f = \phi^a := \phi(a, \cdot)$ acts \textbf{normally hyperbolically} on $M$, meaning that there are constants $A, C > 0$ and a $ d f$-invariant splitting 
    \begin{equation*}
        TM = E_a^{ss}\oplus T\phi\oplus E_a^{uu}
    \end{equation*}
    of the tangent bundle of $M$ such that:
    \begin{itemize}
        \item[(i)] $\Vert \left( d f\rvert_{E^{ss}_a}\right)^n\Vert \leq Ce^{-An}$, for all $n > 0;$
        \item[(ii)]$\Vert \left( d f\rvert_{E^{uu}_a}\right)^n\Vert \leq Ce^{An}$, for all $n < 0.$
    \end{itemize}
    The action $\phi$ is called an \textbf{Anosov action of $\mathbb{R}^q$} if it has an Anosov element $a \in \mathbb{R}^q$.
\end{definition}
    
Anosov elements do not form a discrete set. Each connected component of the set $\mathcal{A}(\phi)$ of all Anosov elements of the action $\phi$ is an open cone in $\mathbb{R}^q$, called a \textbf{chamber} \cite{barbot_transitivity_2011}.
An open convex subcone in a chamber is called a \textbf{regular subcone of the action $\phi$}.
\par Given subsets $K \subset M$, $E \subset \mathbb{R}^q$, we say $K$ is $E$-transitive if $N$ is $\phi$-invariant and there is a point $x$ in $K$ such that 
\begin{equation*}
    \overline{\{\phi(v,x);  v \in E\}} = K,
\end{equation*}
that is, the invariant set $K$ contains a dense $E$-orbit.
    
\begin{proposition}[Spectral Decomposition \cite{barbot_transitivity_2011}]\label{sdbm}
    Let $M$ be a closed smooth manifold and $\phi$ be an Anosov action on $M$. Then there is a finite collection $\{\Lambda_i\}_{i=1}^l$ of pairwise disjoint, locally connected, closed $\phi$-invariant subsets of $M$ such that
    \begin{itemize}
        \item[(i)] For every $1 \leq i\leq l$, the union of closed orbits of $\phi$ inside $\Lambda_i$ is dense in $\Lambda_i$;
        \item[(ii)] For any regular subcone $\mathcal{C}$ the set $\Lambda_i$ is $\mathcal{C}$-transitive;
        \item[(iii)] the non-wandering set of $M$ decomposes 
        \begin{equation*}
            \mathrm{nw}(\phi) = \bigcup_{i=1}^l\Lambda_i.
        \end{equation*}
    \end{itemize}
\end{proposition}
    
\begin{theorem}
    If a contact action $F: \mathbb{R}^q \rightarrow \mathrm{Diff}(M)$ is Anosov, then the underlying foliation satisfies the $\mathrm{SGWC}$.
\end{theorem}
\begin{proof}
    In light of items (i) and (iii) of the Spectral Decomposition \ref{sdbm}, it is sufficient to show that $\mathrm{nw}(F) \neq \emptyset$. To that end, given a point $x \in M$ and an open set $U$ containing $x$, we consider an element $a \in \mathbb{R}^q$ with  $\lvert a \rvert > 1$. The transformation $f = F^a$ is volume preserving (it preserves the volume forms $d M_i$). Hence, by Poincaré's recurrence theorem, there exists a positive natural number $j$ such that $f^j(U) \cap U \neq \emptyset$. But $f^j = F^{ja}$, where, by construction, $\lvert ja \rvert > 1$, which means exactly that $x$ is a non-wandering point. Thus $\mathrm{nw}(F) = M$, and the theorem follows.
    
\end{proof}

\subsection{Riemannian Contact Foliations}\label{secriem}

Following the work of Rukimbira, (cf. \cite{rukimbira_properties_1991,rukimbira_remarks_1993}), we define the following notion
     
\begin{definition}[\textit{Isometric contact foliation}] 
    A foliation $\mathcal{F}$ induced by a $q$-contact structure is called a \textbf{isometric contact foliation} if there is a metric $g$ for which the Reeb vector fields are Killing, that is 
    \begin{equation*}
        \mathcal{L}_{R_i}g = 0.
    \end{equation*}
    In other words, $\mathcal{F}$ is isometric if it is induced by an action of $\mathbb{R}^q$ through isometries.
\end{definition}
     
There is always a metric on the bundle $\mathcal{R}$ for which the Reeb vector fields are Killing, namely 
\begin{equation*}
    g^\tau = \sum_i \lambda_i \otimes \lambda_i,
\end{equation*}
that is, the metric defined by setting $g^\tau(R_i,R_j) = \delta_{ij},$ and extending it linearly. It satisfies $\mathcal{L}_{R_j}g^\tau = \sum_i(\mathcal{L}_{R_j}\lambda_i \otimes \lambda_i + \lambda_i \otimes \mathcal{L}_{R_j}\lambda_i) = 0.$ Therefore, asking for a foliation to be isometric amounts to assert the existence of a transverse invariant metric, that is, a metric $g^\perp$ on the $q$-contact distribution such that $\mathcal{L}_{R_j}g^\perp = 0$ for every Reeb vector field $R_j$. Since every transverse invariant metric is associated with a bundle-like metric \cite[Proposition 3.3]{molino_riemannian_2012}, it follows that every Riemannian contact foliation is, in fact, isometric:
      
 \begin{proposition}\label{Rcontacttransverseproperty}
    For a $q$-contact manifold $(M, \vec{\lambda}, \mathcal{R} \oplus \xi)$ the following are equivalent
    \begin{itemize}
        \item[(i)] The contact foliation $\mathcal{F}$ is isometric;
        \item[(ii)] There exists a metric $g$ on $M$ with respect to which the Reeb vector fields $R_i$ are Killing, and 
        \begin{equation*}
            \lambda_i(X) = g(R_i, X)
        \end{equation*}
        for any vector field on $M$.
    \end{itemize}
\end{proposition}

Following the terminology of \cite{rukimbira_properties_1991}, we define
\begin{definition}[\textit{R-metric}]
    On a $q$-contact manifold $(M, \vec{\lambda}, \mathcal{R} \oplus \xi)$ a metric $g$ satisfying 
    \begin{itemize}
        \item[(i)] $\mathcal{L}_{R_i}g = 0$; 
        \item[(ii)]$g(R_i, \cdot) = \lambda_i$,
    \end{itemize}
    is called an R-metric of the foliation $\mathcal{F}$.
\end{definition} 
    
\begin{remark} 
    Note that Proposition \ref{Rcontacttransverseproperty} can be restated as "a contact foliation is isometric if and only if it admits an R-metric". \textit{From now on, whenever we say the triple $(M, \mathcal{F}, g)$ is an isometric contact foliation, it is implicitly assumed that $g$ is an R-metric.} 
\end{remark}

\begin{example}[The canonical contact structure of $\mathbb{R}^{2n+q}$]
    Consider on $\mathbb{R}^{2n+q}$ the structure of Example \ref{canexm}
    Let $g^\perp$ be the metric on $\xi$ obtained as the restriction of the canonical metric of $\mathbb{R}^{2n+q}$. 
    On $\mathcal{R}$, let $g^\tau = \sum_i \lambda_i \otimes \lambda_i,$ as usual. Then the metric $g^\perp$ is invariant under the Reeb vector fields, as the $\partial_{z_i}$ are Killing vector fields of the canonical metric on $\mathbb{R}^{2n+q}$, and therefore $g:= g^\tau \oplus g^\perp$ satisfies 
    \begin{equation*}
        \mathcal{L}_{\partial_{z_i}}g = 0, 
    \end{equation*}
    that is, $g$ is an R-metric.
    
\end{example}

\begin{example}[Toric extensions]
    Let $(M, \vec{\lambda}, \mathcal{R} \oplus \xi)$ be a $q$-contact manifold admitting an R-metric $g = g^\tau \oplus g^\perp$. 
    Suppose $\pi: E \to M$ is a principal $\mathbb{T}^l$-bundle equipped with a flat connection $TE = H \oplus V$. 
    As it is shown in detail in Example \ref{flatbundleexample}, $E$ admits a $(q+l)$-contact structure determined by $1$-forms 
    \begin{align*} 
        \eta_i =\pi^\ast\lambda_i, &~\text{ for } 1 \leq i \leq q; \\
        \eta_{q+i} = \alpha_i + \pi^\ast\lambda_{j_i}, &~\text{ for } 1 \leq i \leq l, 
    \end{align*}
    where the $\alpha_i$ are $1$-forms that are identically zero on the horizontal bundle and on the vertical bundle are given by 
    \begin{equation*}
    \alpha_i(X_j) = \delta_{ij}, 
    \end{equation*}
    with the $X_i$ being the fields generating the toric action on the fibres.
    The identification of $H$ with $TM$ via $\pi$ induces a splitting $H = \overline{\mathcal{R}} \oplus \widetilde{\xi}$, where $\mathcal{R} \approx \overline{\mathcal{R}}$ and $\xi \approx \widetilde\xi$, both diffeomorphisms being a restriction of $ d\pi$. 
    Thus $TE = \widetilde{\mathcal{R}} \oplus \xi,$  where $\widetilde{\mathcal{R}} = V \oplus \overline{\mathcal{R}}$ is the bundle tangent to the $(q+l)$-dimensional contact foliation $\widetilde{\mathcal{F}}$ on $E$.
    The Reeb vector fields are 
    \begin{equation}\label{liederivativesmetricextension}
        \begin{split}
            \widetilde{R}_i = R_i, &~\text{ for } 1 \leq i \leq q; \\
            \widetilde{R}_{q+i} = \frac{1}{2}(X_i + R_{j_i}), &~\text{ for } 1 \leq i \leq l. 
        \end{split}
    \end{equation}
    We want to show that $(E, \vec{\eta}, TE = \widetilde{\mathcal{R}} \oplus \widetilde{\xi})$ defines an isometric contact foliation for a suitable choice of metric $\widetilde{g}$ on $E$. For this, let $g_0$ be a metric on the tangent bundle $\widetilde{\mathcal{R}}$ defined by
    \begin{equation*} 
        g_0 := \sum_{i=1}^{q+l}\eta_i\otimes\eta_i. 
    \end{equation*}
    This makes $\{\widetilde{R}_i\}$ into an orthonormal basis for $\overline{\mathcal{R}}$. Finally, we let 
    \begin{equation*}
        \widetilde{g} := g_0 \oplus g^\perp
    \end{equation*} 
    be a Riemannian metric on $TE = \widetilde{\mathcal{R}} \oplus \widetilde{\xi}$.
    
    We claim that $\widetilde{g}$ is an R-metric for $E$. First, it is clear from the definition of $\widetilde{g}$ that
    \begin{equation}\label{rmetriccondition1}
    \widetilde{g}(X, \widetilde{R}_i) = g_0(X, \widetilde{R}_i) = \eta_i(X).
    \end{equation}
    It is straightforward that 
    \begin{equation*}
        \left(\mathcal{L}_{\widetilde{R}_i}g_0\right)(\widetilde{R}_j,\widetilde{R}_l) = \mathcal{L}_{\widetilde{R}_i}(\iota_{\widetilde{R}_l}\iota_{\widetilde{R}_j}g_0)-g_0([\widetilde{R}_i,\widetilde{R}_j], \widetilde{R}_l)-g_0(\widetilde{R}_j, [\widetilde{R}_i,\widetilde{R}_l]) = \mathcal{L}_{\widetilde{R}_i}\delta_{ij} = 0,
    \end{equation*}
    hence $\mathcal{L}_{\widetilde{R}_i}g_0 \equiv 0$. Now, from the characterisation of the Reeb fields in \ref{liederivativesmetricextension}, we have
    \begin{equation}\label{liederivativeliftedmetric}
        \begin{split}
            \mathcal{L}_{\widetilde{R_i}}\widetilde{g} = \mathcal{L}_{R_i}\widetilde{g} = \mathcal{L}_{R_i}g_0 + \mathcal{L}_{R_i}g^\perp = 0, &~\text{ for } 1 \leq i \leq q; \\
            \mathcal{L}_{\widetilde{R}_{q+i}}\widetilde{g} = \mathcal{L}_{\widetilde{R}_{q+i}}g^\perp = \frac{1}{2}\mathcal{L}_{X_i}g^\perp + \frac{1}{2}\mathcal{L}_{R_i}g^\perp = \frac{1}{2}\mathcal{L}_{X_i}g^\perp, &~\text{ for } 1 \leq i \leq l.
        \end{split}
    \end{equation}
    Recall, from Example \ref{flatbundleexample}, that for every generator field $X_1$ it holds that
    \begin{equation*}
        [X_i, Z] \text{ is horizontal whenever } Z \text{ is horizontal}.
    \end{equation*}
    Moreover, every vector field $Z$ tangent to $\widetilde{\xi}$ is foliate. Together, these two conditions imply
    \begin{equation*}
        [X_i, Z] \in \overline{\mathcal{R}} \text{ whenever }Z \in \Gamma(\widetilde{\xi}).
    \end{equation*} 
        In particular, for $Y,Z$ tangent to $\widetilde{\xi}$ we have
    \begin{equation*}
    \begin{split}
        (\mathcal{L}_{X_i}g^\perp)(Y,Z) &= \mathcal{L}_{X_i}(\iota_Z\iota_Yg^\perp) - g^\perp([X_i, Y], Z) - g^\perp(Y, [X_i, Z]) \\ 
        &= \mathcal{L}_{X_i}(\iota_Z\iota_Yg^\perp) = 0,
    \end{split}
    \end{equation*}
    as the function $\iota_Z\iota_Yg^\perp$ can be thought as a lift to $E$ of a leaf-wise function of $(M, \mathcal{F})$ to a leaf-wise function of $(E, \widetilde{\mathcal{F}})$. The leaves of $\widetilde{\mathcal{F}}$ are all of the form 
    \begin{equation*}
        \widetilde{\mathcal{F}}(x) \approx \mathbb{T}^l\times \mathcal{F}(\pi(x)),
    \end{equation*}
    with the $X_i$ being vectors on the $\mathbb{T}^l$ directions, hence $\mathcal{L}_{X_i}(\iota_Z\iota_Yg^\perp) = 0$. Thus, Equation \ref{liederivativeliftedmetric} reduces to 
    \begin{equation*}
        \mathcal{L}_{\widetilde{R}_i}\widetilde{g} \equiv 0,    
    \end{equation*}
    which together with \ref{rmetriccondition1} means that $\widetilde{g}$ is an R-metric for $(E, \vec{\eta}, TE = \widetilde{\mathcal{R}} \oplus \widetilde{\xi})$, and $\widetilde{\mathcal{F}}$ is an isometric contact foliation.
    
\end{example}
    
In particular, by taking products of tori with manifolds supporting R-flows, we can produce isometric contact foliations of any dimension $q \geq 2$.
We remark that the resulting manifold is closed if the torus is multiplied by a closed manifold. 
Examples of R-flows include regular and almost regular contact manifolds (the canonical contact structures on odd spheres $S^{2n+1}$ are specific examples of regular contact manifolds, cf. \cite{blair_riemannian_2010, GeigesIntroductionContactTopology2008} for more details and examples), and also every K-contact manifold. 
This last class includes, in particular, compact contact hypersurfaces \cite{banyaga_characteristics_1993}, Kähler manifolds of constant positive holomorphic sectional curvature, and Brieskorn manifolds \cite{kon_structures_1985}. 

\subsubsection{Geometry of manifolds supporting isometric contact foliations}\label{geotopisocon}

The existence of an isometric contact foliation imposes several obstructions to the geometry and topology of the ambient manifold besides orientability. These are primarily consequences of the following theorem.
    
\begin{theorem}\label{harmonicfieldsarebasic} 
    Let $(M, \mathcal{F}, g)$ be an isometric contact foliation on a closed manifold $M$. If a vector field is harmonic with respect to $g$ (that is, if the dual $1$ form $\mu := g(X, \cdot)$ is harmonic), then $X$ is tangent to the $q$-contact distribution $\xi$. In particular, every harmonic $1$-form on $M$ is basic.
\end{theorem}
\begin{proof}
    The coordinate function of $X$ in the direction $R_i$ is 
    \begin{equation*}
        c_i := \lambda_i(X) = g(R_i, X) = \mu(R_i). 
    \end{equation*}
    We wish to show that $c_i \equiv 0$ for every $i$. We divide the proof into two steps: first, we show that if we assume that $c_i \neq 0$ for some value of $i$, then there is no loss of generality in assuming $c_i \neq 0$ for a single value of $i$. Then we show that the latter case cannot happen. 
    
    \textit{Step I: we may assume $\mu(R_j) = c_l\delta_{lj}$.} Since $X$ is harmonic and $R_i$ is a Killing field, it follows that $c_i = \mu(R_i)$ must be a constant (cf. \cite[Proposition 5.13]{PoorWalterA.DifferentialGeometricStructures}). 
    Let us separate the components of $X$ in the $\mathcal{R}$ directions and write $X = \sum_ic_iR_i + \widetilde{X}$, with $\widetilde{X}$ tangent to the  $q$-contact distribution. Taking duals concerning $g$, let us write the harmonic form $\mu$ as a sum 
    \begin{equation}\label{mudecomposed}
        \mu = \sum_{i=1}^qc_i\lambda_i + \widetilde{\mu}.
    \end{equation}
    In particular, $\widetilde{\mu}(R_i) = 0$ for every Reeb vector field by construction.
    As the manifold $M$ is closed, harmonicity implies that $\mu$ is closed. Hence
    \begin{equation}\label{diffmu}
        \sum_{i=1}^qc_i d\lambda_i = -  d\widetilde{\mu}.
    \end{equation}
    Now, around any point $p \in M$, we can choose a chart $(U;x_1, \cdots, x_{2n},z_1, \cdots, z_q)$ as in Proposition \ref{lambdacoordharmonic}. 
    We expand both $\widetilde{\mu}$ and the forms $\lambda_i$ in these coordinates so that $\mu$ is written as
    \begin{equation*} 
        \mu = \sum_{i=1}^q\left(c_i d z_i + \sum_{j=1}^{2n}c_if^i_j(x) d x_j\right) + \underbrace{\sum_{j=1}^{2n}\alpha_j(x,z) d x_j + \sum_{s=1}^q\beta_s(x,z) d z_s}_{\widetilde{\mu}}. 
    \end{equation*}
    By construction, $\widetilde{\mu}(R_i) = 0$ for every $i$. Since in these coordinates $R_i = \partial_{z_i}$, the functions $\beta_s$ on the expansion above are constants equal to zero. Hence
    \begin{equation}\label{mudecomposedcoord}
        \mu = \sum_{i=1}^q\left(c_i d z_i + \sum_{j=1}^{2n}c_if^i_j(x) d x_j\right) + \sum_{j=1}^{2n}\alpha_j(x,z) d x_j,
    \end{equation}
    and, in these coordinates, the equality in \ref{diffmu} becomes 
    \begin{equation*}
    \begin{split}
        \sum_{j,l=1}^{2n}\left( \sum_i c_i\frac{\partial}{\partial x_l}f^i_j(x)  d x_l \wedge  d x_j \right) &= - \sum_{j,l=1}^{2n}\frac{\partial}{\partial x_l}\alpha_j(x,z) d x_l \wedge  d x_j \\ &- \sum_{j=1}^{2n}\sum_{l=1}^q\frac{\partial}{\partial z_l}\alpha_j(x,z) d z_l \wedge  d x_j.
    \end{split}
    \end{equation*} 
    Evaluating both sides of the above equation at the Reeb vector fields, we conclude that the $\alpha_j$ are functions of the coordinate $x$, satisfying
    \begin{equation*}
        - \frac{\partial}{\partial x_l}\alpha_j(x) =  \frac{\partial}{\partial x_l}\sum_i c_if^i_j(x)~\text{ for every } 1\leq l \leq 2n. 
    \end{equation*}
    This means there are constants $K_j$ such that 
    \begin{equation*}
        \alpha_j(x) = - \sum_ic_if^i_j(x) + K_j.
    \end{equation*}
    Applying this to Equation \ref{mudecomposedcoord} gives 
    \begin{equation*}
        \mu = \sum_{i=1}^qc_i d z_i + \sum_{j=1}^{2n}K_j d x_j.
    \end{equation*}
    Choose an index $l$ such that $c_l \neq 0$, and write 
    \begin{equation*} 
        \mu - \sum_{i \neq l}c_i d z_i = c_l\lambda_l + \sum_{j=1}^{2n}K_j d x_j.
    \end{equation*}
    The RHS is a harmonic form, for each $ d z_i$ is harmonic (cf. Proposition \ref{lambdacoordharmonic}). Thus the $1$-form 
    \begin{equation*}
        \mu_0 := c_l\lambda_l + \sum_{j=1}^{2n}K_j d x_j
    \end{equation*}
    is a harmonic form on $M$ satisfying
    \begin{equation*}
        \mu_0(R_j) = c_l\delta_{jl},
    \end{equation*}
    as we wanted. 

    \textit{Step II: given a constant $c_l \neq 0$, there can be no harmonic form $\mu$ with $\mu(R_j) = c_l\delta_{lj}$.} Let us assume the existence of such a form. To simplify things a bit, let the non-zero coefficient be $c_1$. Let $X = c_1R_1 + \widetilde{X}$ be the dual vector field of $\mu$, where $\widetilde{X}$ is tangent to $\xi$. Then 
    \begin{equation*}
        \eta := \lambda_1 - \frac{1}{c_1}\mu
    \end{equation*}
    is a basic $1$-form on $M$ such that $ d\eta =  d\lambda_1$. This follows because, as $\mu(R_j) = c_1\delta_{1j}$, we have 
    \begin{equation*}
        \eta(R_i) = 0 \text{ for every }i.
    \end{equation*}
    Moreover, since $\mu$ is harmonic, $\mathcal{L}_Y\mu = 0$ for every Killing field $Y$ (cf \cite[Proposition 5.13]{PoorWalterA.DifferentialGeometricStructures}). In particular,
    \begin{equation*}
        \mathcal{L}_{R_i}\eta = \mathcal{L}_{R_i}\lambda_1 + \mathcal{L}_{R_i}\mu + \left(R_i\frac{1}{c_1}\right)\mu = 0 \text{ for every }i,
    \end{equation*}
    so that $\eta$ is indeed basic. Since $M$ is a closed manifold, it follows that $\mu$ is closed, hence $ d\eta =  d\lambda_1$. Consequently, we have
    \begin{equation}\label{eq01obstcur}
        d(\eta\wedge( d\lambda_1)^{n-1}) =  d\eta\wedge( d\lambda_1)^{n-1} = ( d\lambda_1)^{n}.
    \end{equation}
    Using equation \ref{eq01obstcur} above we obtain a volume form on $M$ written as
    \begin{equation*}
    \begin{split}
       dM_1 &= \lambda\wedge d(\eta\wedge( d\lambda_1)^{n-1})\\
        &= (-1)^q d\lambda\wedge\eta\wedge( d\lambda_1)^{n-1}+ (-1)^{q+1} d(\lambda\wedge\eta\wedge( d\lambda_1)^{n-1}).
    \end{split}
    \end{equation*}
    Since $d\lambda\wedge\eta\wedge( d\lambda_1)^{n-1}$ is a basic form, we have $ d\lambda\wedge\eta\wedge( d\lambda_1)^{n-1} \equiv 0$. Hence 
    \begin{equation*}
        \lambda\wedge( d\lambda_1)^{n} = (-1)^{q+1} d(\lambda\wedge\eta\wedge( d\lambda_1)^{n-1}),
    \end{equation*}
    which can not be, for a closed manifold admits no exact volume form.
    
    We conclude that there can be no harmonic vector field satisfying $X = c_1R_1 + \widetilde{X}$, with $\widetilde{X}$ tangent to $\xi$, which implies, in turn, in light of step (I), that any harmonic vector field with respect to an R-metric for a contact foliation $\mathcal{F}$ must be transverse to that foliation, as we wanted. In particular, since each $R_i$ is Killing, any harmonic $1$-form must satisfy
    \begin{equation*}
    \begin{cases}
        \iota_{R_i}\mu = 0 \text{ for every } 1 \leq i \leq q; \\
        \mathcal{L}_{R_i}\mu = 0 \text{ for every } 1 \leq i \leq q,
    \end{cases}
    \end{equation*}
    hence every harmonic $1$-form is basic.
    
\end{proof}
    
For any foliation $\mathcal{F}$ there is an injection $H^1_B(\mathcal{F}) \hookrightarrow H^1(M; \mathbb{R})$ of the first basic cohomology group into the first real De Rham cohomology (cf. \cite[Proposition 4.1]{TondeurGeometryFoliations1997}. Using our last result, we can show that for Riemannian contact foliations this is an isomorphism.
    
\begin{proposition}
    If $(M, \mathcal{F}, g)$ is a closed manifold supporting an isometric contact foliation, there is an isomorphism
    \begin{equation*}
        H^1(M; \mathbb{R}) \approx H^1_B(\mathcal{F})
    \end{equation*}
    between the first De Rham cohomology group of $M$ and the first basic cohomology group of the foliation $\mathcal{F}$.
\end{proposition}
\begin{proof}
    Hodge's Theorem states that every class in $H^1(M;\mathbb{R})$ has a unique harmonic representative, and since every harmonic $1$-form is basic, this gives an injection from $H^1(M; \mathbb{R})$ into $H^1_B(\mathcal{F})$, which implies the desired isomorphism.
    
\end{proof}
    
We apply this result in order to exclude the existence of isometric contact foliations on a number of closed manifolds:
    
\begin{proposition}\label{cohomologyobst} 
    Let $M$ be a closed orientable manifold of dimension, and suppose that the dimension of $H^1(M; \mathbb{R})$ as a real vector space is the same as the manifold dimension of $M$. Then $M$ does not support isometric contact foliations.
\end{proposition}
\begin{proof}
    Given $q \geq 1$, write $\dim M = 2n+q$ and suppose, by contradiction, that $\mathcal{F}$ is an isometric contact foliation on $M$. Let $\eta_1, \cdots, \eta_{2n+q}$ be harmonic forms whose cohomology classes form a basis for the $\mathbb{R}$-vector space $H^1(M; \mathbb{R})$. On the one hand, since the classes $[\eta_i]$ are linearly independent, the product $\eta_1\wedge\cdots\wedge\eta_{2n+q}$ is non-vanishing. Hence its class in the top cohomology group $H^{2n+q}(M; \mathbb{R})$ can also be represented as a volume form $\omega$ on $M$. This means, in particular, that we can find a $(2n+q-1)$-form $\eta$ such that 
    \begin{equation*}
        \eta_1\wedge\cdots\wedge\eta_{2n+q} = \omega -  d\eta. 
    \end{equation*}
    On the other hand, since $\mathcal{F}$ is an isometric contact foliation, it follows from Theorem \ref{harmonicfieldsarebasic} that each one of the $1$-forms $\eta_i$ is also basic, hence $\eta_1\wedge\cdots\wedge\eta_{2n+q} = 0$ (the maximum rank of basic form in $H_B^\ast(\mathcal{F})$ being the codimension $2n$, which is less than $2n+q$), and therefore the volume form $\omega =  d\eta$ is exact, a contradiction.
    
\end{proof}

\begin{remark}
Consider a linear foliation of $\mathbb{R}^{2n+q}$ by $q$-planes of rational ``inclination", such that its projection $\mathcal{F}$ to the torus $\mathbb{T}^{2n+q}$ is a foliation by tori $\mathbb{T}^q$. Then $\mathcal{F}$  is a Riemannian foliation of $\mathbb{T}^{2n+q}$ when the latter is equipped with the flat metric coming from the euclidean metric of $\mathbb{R}^{2n+q}$. However, it is not a contact foliation since it admits global sections. More generally, the following corollary to Proposition \ref{cohomologyobst} guarantees no torus supports isometric contact foliations.
\end{remark}
    
\begin{corollary}
    There can be no isometric contact foliation on the torus $\mathbb{T}^{2n+q}$, whichever is $q \geq 1$. More generally, if $p: E \to \mathbb{T}^m$ is a $\mathbb{T}^l$-bundle, then $E$ supports no isometric contact foliation.
\end{corollary}
\begin{proof}
    For a torus $\mathbb{T}^{2n+q}$ the first cohomology group is isomorphic to $\mathbb{R}^{2n+q}$. Hence it has the same dimension as the manifold, and the conclusion follows.
    
    As for a torus bundle $p: E \to \mathbb{T}^m$. It is a matter of noticing that the exact homotopy sequence of the fibration $\mathbb{T}^l \overset{i}{\hookrightarrow} E \xrightarrow[]{p} \mathbb{T}^m$ becomes 
    \begin{equation*}
        \cdots 0 \rightarrow \mathbb{Z}^l \overset{i_\ast}{\hookrightarrow} \pi_1(E) \xrightarrow[]{p_\ast} \mathbb{Z}^m,
    \end{equation*}
    hence 
    \begin{equation*}
        ^{\textstyle \pi_1(E)}\!\big/_{\textstyle \mathbb{Z}^l} \approx \mathbb{Z}^m.
    \end{equation*}
    If $\{e_i\}$ and $\{f_i\}$ are basis for $\mathbb{Z}^l$ and $\mathbb{Z}^m$, respectively, we set $E_i := i_\ast(e_i)$, and choose $E_{n+j}$ such that the class $E_{n+j} + \mathbb{Z}^l$ corresponds to $f_j$, thus obtaining a basis $\{E_1, \cdots, E_{n+m}\}$ for the fundamental group $\pi_1(E)$. Hence 
    \begin{equation*}
        \pi_1(E) \approx \mathbb{Z}^{m+n}.
    \end{equation*}    
    Now, the first de Rham cohomology group is isomorphic to $\mathrm{Hom}(\pi_1(E); (\mathbb{R},+))$ via 
    \begin{equation*} 
        [\omega] \mapsto \left( [\gamma] \mapsto \int_\gamma\omega\right). 
    \end{equation*}
    Since $\mathrm{Hom}(\pi_1(E); (\mathbb{R},+))$ is generated (with $\mathbb{R}$ coefficients) by the mappings sending each generator $E_i$ to $1$, it follows that 
    \begin{equation*}
        H^1(E;\mathbb{R}) \approx \mathrm{Hom}(\pi_1(E); (\mathbb{R},+)) \approx \mathbb{R}^{m+n}, 
    \end{equation*}
    and therefore $\dim H^1(E;\mathbb{R}) = m+n = \dim E$. According to Proposition \ref{cohomologyobst}, there can be no isometric contact foliation on $E$.
    
\end{proof}
    
Another consequence of Theorem \ref{harmonicfieldsarebasic} is the following theorem, which imposes geometric restrictions on the metric $g$, making a contact foliation $\mathcal{F}$ isometric.
    
\begin{theorem}\label{curverest}
    If $(M, \mathcal{F}, g)$ is an isometric foliation on a closed manifold $M$, then the sectional curvature of $g$ is neither strictly positive nor strictly non-positive. In other words, there are points $p,q \in M$ and planes $\pi \in T_pM$ and $\sigma \in T_qM$ such that 
    \begin{align*}
        \mathrm{sec}(\pi) &> 0 \\
        \mathrm{sec}(\sigma) &\leq 0.
    \end{align*}

\end{theorem}
\begin{proof}
    We argue by contradiction. Assume that $g$ is an R-metric whose sectional curvature is non-positive, and denote by $\nabla$ the Levi-Civita connection. Since the action is locally free, for small enough $t$, the flow $\phi_1^s$ of $R_1$ provides a homotopy between a non-trivial isometry $\phi^t$ and the identity $\mathrm{id} = \phi_1^0$. It follows that $M$ admits a \textit{non-vanishing} parallel vector field $X$ \cite[Proposition 2]{h_blaine_lawson_compact_1972}. Being a parallel field on a compact Riemannian manifold, $X$ must be Killing, and harmonic \cite[Proposition 5.12]{PoorWalterA.DifferentialGeometricStructures}. It follows from Theorem \ref{harmonicfieldsarebasic} that $X$ is everywhere tangent to the $q$-contact distribution, and its dual form $\mu = g(X,\cdot)$ is basic. We claim this implies $[R_i,X]=0$ for every $i$. Indeed, given $Y \in \Gamma(M)$, we have
    \begin{align*}
        0 = \mathcal{L}_{R_i}\mu(Y) &= R_i\mu(Y) -\mu([R_i, Y]) \\
        &= R_ig(X,Y) - g(X, [R_i, Y]) \\
        &= \mathcal{L}_{R_i}g(X,Y) - g(X, [R_i, Y]) + g([R_i, X], Y) + g(X, [R_i, Y])\\
        &= g([R_i, X], Y),
    \end{align*}
    from where it follows $[R_i,X]=0$, as we wanted.
    Moreover, since $X$ is also Killing, for arbitrary $Y \in \Gamma(M)$ and $1\leq i \leq q$, we have
    \begin{align*}
        \mathcal{L}_{X}\lambda_i(Y) &= X\lambda_i(Y) -\lambda_i([X, Y]) \\
        &= Xg(R_i,Y) - g(R_i, [X, Y]) \\
        &= \mathcal{L}_{X}g(R_i,Y)  + g([X, R_i], Y) \\
        &= g(R_i, [X, Y]) - g(R_i, [X,Y]) \\
        &= 0.
    \end{align*}
    Finally, since $\lambda_i(X) = \mathcal{L}_X\lambda_i = 0$, it follows
    \begin{equation*} 
        0 = \mathcal{L}_X\lambda_i =  d \iota_X\lambda_i + \iota_X d\lambda_i = \iota_X d\lambda_i,
    \end{equation*}
    that is, $X$ belongs to the kernel of $ d\lambda_i$, namely $\mathcal{R}$. This would imply $X = 0$, which contradicts the non-vanishing of $X$ as given by \cite[Proposition 5.12]{PoorWalterA.DifferentialGeometricStructures}, proving the corollary.
    
    On the other hand, if the sectional curvature of $g$ is strictly positive, then it follows from \cite[Theorem 8.3.5]{petersen_riemannian_2006} that any two commuting Killing fields on $M$ must be linearly dependent somewhere. We know, however, that $\{R_i\}$ is a global frame of commuting Killing fields for $\mathcal{R}$. Hence $g$ can not have positive sectional curvature.
\end{proof}

\subsubsection{Compatible Toric Actions}\label{comptoract}

 Intending to find closed orbits, we define a particular class of toric actions on a $q$-contact manifold. Adapting arguments from Banyaga and Rukimbira \cite{banyaga_characteristics_1995}, we show that contact manifolds supporting such actions satisfy the $\mathrm{WGWC}$. Later, we show that every isometric contact foliation supports a compatible action of $\mathbb{T}^q$ by isometries. The existence of this compatible action will allow us, by studying the closures of leaves, to conclude that in the class of isometric contact foliations, the $\mathrm{WGWC}$ implies the $\mathrm{SGWC}$.
    
\begin{definition}[\textit{Compatible toric action}] 
    Let $(M, \vec{\lambda}, \mathcal{R} \oplus \xi)$ be a contact foliation, and  $\mathbb{T}^l \to \mathrm{Diff}(M)$ be an action of the torus $\mathbb{T}^l$ on $M$. We will say this action is \textbf{compatible} with the contact foliation $\mathcal{F}$ if it preserves the non-degenerate forms $ d\lambda_i$.  
    More precisely, let $X_1, \cdots, X_l$ be infinitesimal generators of the toric action. The action is compatible if
    \begin{equation*}
        \mathcal{L}_{X_i} d\lambda_j = 0 ~\text{ for any } i, j.
    \end{equation*}
\end{definition}

The most critical example of contact foliations supporting compatible toric actions is the class of isometric contact foliations.
        
\begin{proposition}\label{compactisoR} 
    If $\mathcal{F}$ is an isometric contact foliation on a compact manifold $M$, then $\mathcal{F}$ admits a compatible action by isometries $\mathbb{T}^q \rightarrow \mathrm{Iso}(M).$ Moreover, the action of $\mathbb{T}^q$ and the contact action of $\mathbb{R}^q$ commute with one another.
\end{proposition}
\begin{proof}
    Due to Proposition \ref{Rcontacttransverseproperty} we can assume without loss of generality that $g$ is an R-metric, that is, that $\lambda_i(X) = g(R_i, X)$. The contact action of $\mathbb{R}^q$ induces a representation of $\mathbb{R}^q$ in the group $\mathrm{Iso}(M).$ Let $\mathrm{E}(M, \vec{\lambda}) \subset \mathrm{Iso}(M)$ be the \emph{enveloping group} of the action, that is, the closure of the subgroup spanned by the action representation:
    \begin{equation*}
        \mathrm{E}(M, \vec{\lambda}) := \overline{\mathrm{Span}\left\{\bigcup_{i=1}^q\bigcup_{t\in\mathbb{R}}\phi^t_i\right\}},
    \end{equation*}
    where $\phi_i^t$ is the time $t$ flow map of the Reeb vector field $R_i$, as usual.
    By the famous Myers-Steenrod theorem \cite{myers_group_1939}, $\mathrm{Iso}(M)$ is a compact Lie group, and $\mathrm{E}(M, \vec{\lambda})$ is, therefore, the smallest compact Abelian subgroup containing the image of $\mathbb{R}^q$, hence $\mathrm{E}(M, \vec{\lambda})$ is a torus $\mathrm{E}(M, \vec{\lambda}) \approx \mathbb{T}^l$, with $l \geq q$. 
    Thus $\mathbb{T}^q \hookrightarrow \mathrm{E}(M, \vec{\lambda})$ induces an $\mathbb{T}^q$-action on $M$. By construction, this is an isometric action, so its generators $X_1, \cdots, X_q$ are Killing vector fields, that is, $\mathcal{L}_{X_i}g = 0$ for every $i$.
     
    We claim this action is compatible. Indeed, since all the elements in $\mathrm{E}(M, \vec{\lambda})$ commute, for any element $h \in \mathbb{T}^q$ we have $h\circ\phi_i^t = \phi_i^t\circ h$. This implies the action preserves the Reeb vector fields, that is, $ d h_p (R_i\rvert_p) = R_i\rvert_{h(p)}$. Therefore the action preserves the adapted coframe:
    \begin{align*}
        h^\ast\lambda_i\rvert_p(X) &= g_{h(p)}(R_i\rvert_{h(p)},  d h_p(X) )\\
        &= h^\ast g_p(R_i, X) \\
        &= \lambda_i\rvert_p(X).
    \end{align*}
    Moreover, from the naturality of the pullback operation
    \begin{equation*}
        h^\ast  d\lambda_i =  d\lambda_i,
    \end{equation*}
    so that the action is compatible, as we wished.
     
\end{proof}
        
Compatible toric actions are helpful when searching for closed orbits because whenever a generator $X_i$ is tangent to a leaf $\mathcal{F}(x)$, the leaf in question is not a plane.
        
\begin{proposition}\label{tangencycomp}
    Let $(M, \vec{\lambda}, \mathcal{R} \oplus \xi)$ be a closed $q$-contact manifold, and  $\mathbb{T}^s \to \mathrm{Diff}(M)$ be a compatible action with generators $X_1, \cdots, X_s$. Then each generator is tangent to $\mathcal{R}$ in at least two points. 
\end{proposition}
\begin{proof}
    We consider on $M$ a $q$-form $\overline{\lambda}$, obtained from the leaf-wise volume element $\lambda := \lambda_1\wedge\cdots\wedge\lambda_q$ by averaging it with respect to a Haar measure $\mu$ on $\mathbb{T}^s$, in the following sense: 
    \begin{equation*}
        \overline{\lambda}_p(Y_1, \cdots, Y_q) := \int\displaylimits_{g \in \mathbb{T}^s}g^\ast\lambda_p(Y_1, \cdots, Y_q)  d\mu(g).
    \end{equation*}
    Note that $g$ is being used to write both the element of $\mathbb{T}^l$ and the diffeomorphism of $M$ it represents. Due to the invariance of $\mu$ under multiplication on $\mathbb{T}^s$, the form $\overline{\lambda}$ is invariant under the action of $\mathbb{T}^s$. In particular, for any $i \leq q$, any $l \leq s$ and any set of vector fields $Y_1, \cdots, Y_l$ we have 
    \begin{equation*}
        \mathcal{L}_{X_i}(\iota_{Y_1}\cdots\iota_{Y_l}\overline{\lambda}) = 0.
    \end{equation*}
    We also define $\overline{\lambda}_i$ by
    \begin{equation*}
        \overline{\lambda}_i\rvert_p(Y) := \int\displaylimits_{g \in \mathbb{T}^q}g^\ast\lambda_i\rvert_p(Y)  d\mu(g).
    \end{equation*}
    We consider, for each $i \leq s$ and $j\leq q$, the functions 
    \begin{equation*}
        \phi_{ij}(p):= (-1)\iota_{X_i}\overline{\lambda}_j\rvert_p,
    \end{equation*}
    Note that since $g^\ast X_i = X_i$ for any $g \in \mathbb{T}^q$ and generator $X_i$, these functions are constant along the orbits of the toric action. We claim that 
    \begin{equation}\label{diffsimples} 
        d\phi_{ij} = \iota_{X_i} d\lambda_j,   
    \end{equation}
    To see this, begin defining auxiliary forms $\alpha_j = \lambda_j - \overline{\lambda}_j$, and notice that since the action preserves $d\lambda_j$, the forms $\alpha_j$ is closed. Together with $\mathcal{L}_{X_i}\overline{\lambda}$, the closeness of $\alpha_j$ yields, for any generator $X_i$:
    \begin{equation*}
        \iota_{X_i} d\lambda_j +  d\iota_{X_i}\lambda_j = \mathcal{L}_{X_i}\lambda_j = \mathcal{L}_{X_i}\alpha_j + \mathcal{L}_{X_i}\overline{\lambda}_j =  d\iota_{X_i}\alpha_j =  d\iota_{X_i}\lambda_j -  d\iota_{X_i}\overline{\lambda}_j,
    \end{equation*}
    that is
    \begin{equation*}
        d \phi_{ij} = - d\iota_{X_i}\overline{\lambda}_j = \iota_{X_i} d\lambda_j, ~~i,j = 1, \cdots, q,
    \end{equation*}
    which proves \ref{diffsimples}.
    Moreover, for any $i, j, l$, the relation $\iota_{[R_i, X_j]} = [\mathcal{L}_{R_i}, \iota_{X_j}]$, together with Equation \ref{diffsimples}, yields
    \begin{align*}
        \iota_{[R_i, X_j] d\lambda_l} &= \mathcal{L}_{R_i}\iota_{X_j} d\lambda_l \\
        &= \iota_{R_i} d(\iota_{X_j} d\lambda_l) +  d(\iota_{R_i}\iota_{X_j} d\lambda_l) \\
        &= \iota_{R_i} d  d\phi_{ij}\\
        &=0.
    \end{align*}
    Hence $[R_i, X_j] \in \ker d\lambda_l = \mathcal{R}$, that is, the generators $X_i$ are \textit{foliate vector fields} of $\mathcal{F}$. In particular, their flows preserve the leaves of $\mathcal{F}$. 
    Finally, consider the function $\phi_{ij}$. Since  $M$ is closed, $\phi_{ij}$ has at least two critical points, where $\phi_{ij}$ attains its maximum and minimum, respectively. At a critical point $p$, one has
    \begin{equation*}
        0 =  d\phi_{ij}\rvert_p =  d\lambda_j\rvert_p(X_i\rvert_p, \cdot) \iff X_i\rvert_p \in \mathcal{R}_p,
    \end{equation*}
    that is, the critical points of $\phi_{ij}$ are exactly the points where $X_i$ is tangent to $\mathcal{R}.$
    
\end{proof}
    
\begin{theorem}\label{noplanefol}
    If a closed $q$-contact manifold admits a compatible toric action, then the contact foliation can not be a foliation by planes. In other words, closed $q$-contact manifolds supporting compatible toric actions satisfy the $\mathrm{WGWC}$.
\end{theorem}
\begin{proof}
 

As we saw in the last proposition, the generating fields $X_i$ are foliate. Recall that the set of foliate vector fields is the normaliser 
    \begin{equation*}
        \mathfrak{L}(\mathcal{F}) := \{ X \in \Gamma(M); [X,Y] \in \Gamma(\mathcal{F}) \text{ for all } Y \in \Gamma(\mathcal{F})\}
    \end{equation*} of the Lie algebra $\Gamma(\mathcal{F})$ of vector fields tangent to $\mathcal{F}$, that is, sections of $\mathcal{R}$. The Lie algebra of sections of $\mathcal{R}$ is clearly a sub-algebra of $\mathfrak{L}(\mathcal{F})$, and the quotient 
    \begin{equation*}
        \mathfrak{t}(\mathcal{F}) := ~^{\textstyle \mathfrak{L}(\mathcal{F})}\!\big/_{\textstyle \Gamma(\mathcal{F})}
    \end{equation*}
    is the Lie algebra of \textit{transverse vector fields.} We denote by $\overline{X}$ the image of a field $X$ under the projection $\mathfrak{L}(\mathcal{F}) \to \mathfrak{t}(\mathcal{F})$. Each transverse field acts on the leaf space $M/{\mathcal{F}}$ via its flow. According to Proposition \ref{tangencycomp} above, if $p$ is a critical point of the function $\phi_{ij}$, then the generating field $X_i$ is tangent to $\mathcal{F}$ at $p$. Consequently, $\overline{X}_i$ is zero at the point (that is, the \emph{leaf}) $\mathcal{F}(p)$. This means the flow of $X_i$ fixes the entire leaf $\mathcal{F}(p)$. In particular, the entire orbit of $p$ under $X_i$ is contained in $\mathcal{F}(p)$.

Finally, since we can choose the generators of the toric action to be all periodic fields, it follows that $\mathcal{F}(p)$ contains an essential closed curve, and it is, therefore, not a plane.
   
\end{proof}
    
\noindent In particular, for $q=1$ Theorem \ref{noplanefol} is a result on circle invariant pre-symplectic forms by Banyaga and Rukimbira \cite{banyaga_characteristics_1995}.
    
\begin{corollary}
    Every isometric contact foliation on a closed manifold satisfies the $\mathrm{WGWC}$. 
\end{corollary}
\begin{proof}
    This is just Theorem \ref{noplanefol} restated in light of Proposition \ref{compactisoR}.
    
\end{proof}
    
Now we generalise a result from \cite{banyaga_characteristics_1995}, proving that an isometric contact action of $\mathbb{R}^q$ must have at least two closed orbits. To that end, we first need the following Proposition, due to Caramello Jr., which we state here for completeness:
    
\begin{proposition}[Proposition 3.1 in \cite{caramello_jr_positively_2019}]\label{caramello}
    Let $(M, \mathcal{F})$ be a Riemannian foliation and $\overline{X}$ be a transverse Killing vector field. Then each connect component $N$ of the set of zeros of $\overline{X}$ is an even codimensional closed submanifold of $M$ saturated by the leaves of $\mathcal{F}$. Moreover, $N$ is horizontally totally geodesic, and if $\mathcal{F}$ is transversely orientable, so is $(N, \mathcal{F}\rvert_N).$
\end{proposition}
    
This is a generalisation to Riemannian foliations of a result for Riemannian manifolds regarding zero sets of Killing fields, that is, sets of points where the vector field vanishes (cf. \cite[Theorem 5.3, Chapter II]{kobayashi_transformation_1995}, and also \cite{kobayashi_fixed_1958}). We will need the following simple lemma as well.
    
\begin{lemma}\label{restevencod} 
    Let $(M, \vec{\lambda}, \mathcal{R} \oplus \xi)$ be a contact foliation. If $N \subset M$ is an $\mathcal{F}$-saturated submanifold of even codimension, then $(N, \vec{\lambda}\rvert_N, \mathcal{R}\rvert_N \oplus \xi \cap TN)$ is a $q$-contact structure on $N$.
\end{lemma}
\begin{proof}
    First note that, since $N$ is saturated (that is, $\mathcal{F}(p) \subset N$ for every $p \in N$), we have $\dim N \geq q$, and therefore $\mathrm{codim} N = 2m$ for some $m \leq n$. Moreover, all Reeb vector fields are everywhere tangent to $N$, hence $\mathcal{R}\rvert_N \subset TN$. 
    One can also check that
    \begin{equation*}
        \ker\left(\lambda_i\rvert_N\right) = \ker\lambda_i \cap TN,
    \end{equation*} 
    for any $i$ and, consequently, 
    \begin{equation*}
        \bigcap_i\ker\left(\lambda_i\rvert_N\right) = \xi \cap TN.
    \end{equation*}
    Moreover, the derivatives satisfy
    \begin{equation*}
        \ker\left( d\lambda_i\rvert_N\right) = \mathcal{R} \cap TN = \mathcal{R}\rvert_N,
    \end{equation*}
    for whichever $i$ we choose. Hence
    \begin{equation*}
    \begin{split}
        TN &= (TN \cap \mathcal{R}) \oplus (TN \cap \xi) \\
        &= \mathcal{R}\rvert_N \oplus (TN \cap \xi) \\
        &= \ker\left( d\lambda_i\rvert_N\right) \oplus \left( \bigcap_j\ker\left(\lambda_j\rvert_N\right)\right),
    \end{split}
    \end{equation*} 
    and therefore $(N, \lambda_1\rvert_N, \cdots, \lambda_q\rvert_N, \mathcal{R}\rvert_N \oplus \xi \cap TN)$ is a $q$-contact structure on $N$, as claimed.

\end{proof}

\begin{theorem}\label{closedorbitsriemanniancase}
    Every isometric contact foliation $\mathcal{F}$ on a closed manifold $M$ satisfies the $\mathrm{SGWC}$. Moreover, $\mathcal{F}$ has at least $2$ closed orbits; and the set $\mathcal{C}_\mathcal{F}$ of closed orbits of $\mathcal{F}$ consists of a union 
    \begin{equation*}
        \mathcal{C}_\mathcal{F} = \bigcup N_i,
    \end{equation*}
    where each $N_i$ is an even-codimension totally geodesic closed submanifold of $(M,g)$ carrying a closed contact foliation, that is, one where every leaf is closed.
\end{theorem}
\begin{proof}
    Let $\mathbb{T}^q \rightarrow \mathrm{Iso}(M)$ be a compatible action by isometries as in Proposition \ref{compactisoR}. As remarked before in Proposition \ref{tangencycomp}, since the generators $X_1, \cdots, X_q$ of the compatible toric action are foliated vector fields, to find a close orbit, it is sufficient to find a point where all the generators are tangent to $\mathcal{R}$. The main idea here is that the set of points where these generators are tangent to the foliation can be realised as a zero set of a Killing vector field. Hence it has even codimension and therefore carries a $q$-contact structure, according to Lemma \ref{restevencod}. 
    
    Given a generator $X_i$, we denote by $\overline{X}_i$ its image under the projection onto the algebra of transverse fields. By construction, each $\overline{X}_i$ is \textit{a transverse Killing vector field}, that is, one that preserves the metric $g^\perp$ on $\xi$.
    
    Consider the set of zeros of $\overline{X}_1$, i.e., 
    \begin{equation*}
        \{p \in M; \overline{X}_1\rvert_p = 0\} \subset M.
    \end{equation*}
    This is precisely the set of points $p \in M$ where $X_1$ is tangent to the foliation, and, due to Proposition \ref{tangencycomp}, it has at least two elements. Let $M_1$ be a connected component of this set. 
    It follows from Proposition \ref{caramello} that $M_1$ is a closed submanifold of even codimension saturated by $\mathcal{F}$. According to Lemma \ref{restevencod}, $M_1$ admits a $q$-contact structure inducing on $M_1$ the same foliation as $\mathcal{R}\rvert_{M_1}$, and whose Reeb vector fields are $R_i\rvert_{M_1}$. 
    By construction, $X_1$ is everywhere tangent to the leaves on $M_1$. Now, since $M_1$ is closed, the function $\phi_{2j} = -\lambda_j(X_2)$ has a critical point on $M_1$, and hence the set 
    \begin{equation*}
        \{p \in M_1; \overline{X}_2\rvert_p = 0\}
    \end{equation*} is non-empty. We choose a connected component $M_2$ on this set. It follows again from proposition \ref{caramello} that $M_2$ is a closed submanifold of even codimension (both in $M$ and $M_1$), which is saturated by $\mathcal{F}$ and on which $X_2$ restricts to a vector field everywhere tangent to the leaves. 
    
    Iterating this procedure, we eventually arrive at an $\mathcal{F}$-saturated submanifold $M_{q-1}$, of even codimension, on which $X_1, \cdots, X_{q-1}$ are everywhere tangent to the leaves, by construction. 
    On this closed submanifold, the functions $\phi_{q,j}$ have at least two distinct critical points $p$ and $q$, corresponding to the extrema of the function, so that both $X_q\rvert_{p}$ and $X_q\rvert_{q}$ are tangent to $\mathcal{R}$. 
    By construction, all the $X_i$ are tangent to $\mathcal{R}$ at $p$ and $q$, hence, by Proposition $\ref{tangencycomp}$, the leaves $\mathcal{F}(p)$ and $\mathcal{F}(q)$ are closed.
    
    This proves the subset $\mathcal{C}_\mathcal{F}$ of points in $M$ belonging to a closed leaf is non-empty. It is known since the work of Myers and Steenrod \cite{myers_group_1939} that the closures of orbits are the orbits of the enveloping group $\mathrm{E}(M, \vec{\lambda})$ (cf. Proposition \ref{compactisoR}). 
    Thus, if $p \in \mathcal{C}_{\mathcal{F}}$, then $p$ belongs to a closed leaf $\mathcal{F}(p) = \overline{\mathcal{F}(p)}$, which is the orbit of $p$ under the action of $\mathrm{E}(M, \vec{\lambda})$. 
    Since $\mathcal{F}(p)$ is a torus $\mathbb{T}^q$, $p$ is a fixed point of some element of $\mathrm{E}(M, \vec{\lambda})$. Hence it is a fixed point of the elements of an entire isotropy subgroup of $\mathrm{E}(M, \vec{\lambda})$ (as is the entire leaf $L$ to which $p$ belongs). Now, being a compact Abelian Lie group, $\mathrm{E}(M, \vec{\lambda})$ admits only a finite number of such isotropy groups, say $I_1, \cdots, I_l$. Denote by $\mathfrak{I}_1, \cdots, \mathfrak{I}_l$ the corresponding Lie algebras of Killing fields. The leaf $\mathcal{F}(p)$ is a set of collective zeros of the Killing fields in one of the algebras $\mathfrak{I}_j$. Hence it follows from \cite[Corollary 1]{kobayashi_fixed_1958} that $\mathcal{F}(p)$ is contained in an even-codimensional totally geodesic submanifold $N$ of $M$. The finite collection of all these submanifolds gives us the desired partition. Moreover, as pointed out before, $p \in N$ implies $\mathcal{F}(p) \subset N$. Hence each $N$ is a saturated submanifold of even codimension, so that Lemma \ref{restevencod} applies, allowing us to conclude that the restriction of $\vec{\lambda}$ to $N$ gives rise to a contact foliation, the leaves of which are all closed by construction. 
    
\end{proof}

\section{Conclusion}
The theory of $q$-contact structures, as dealt with here, started with Almeida's work \cite{almeida_contact_2018}. 
He focused mainly on the contact action, and at that, in the particular case when such action is partially hyperbolic. In this work, we focus instead on the perspective of the \textit{contact foliation}. We dived into the fundamental theory of contact foliations, greatly expanding the foundations laid out by Almeida.
In particular, we provided new examples, local descriptions for these foliations (Propositions \ref{lambdacoord} and \ref{lambdacoordharmonic}), canonical forms for the adapted coframe in particular cases (Propositions \ref{localRepresentationCodimension2} and \ref{localRepresentationUniformCase}), and pointed out some basic properties of the holonomy groups of various distinct foliations related to a $q$-contact structure (Proposition \ref{holonomyproperties}). 
\par Two novel results of ours are the reduction procedure of Theorem \ref{contredux} and the characterisation of uniform contact manifolds in Theorem \ref{uniformStructuresFibrations}.
The former is a partial converse to the construction of contact structures in flat toric bundles introduced by Almeida.
It allowed us to prove that non-trivial minimal contact foliations can not take place if any of the Weinstein conjectures are valid (cf. Theorem \ref{equivconjmin}), a result not all obvious at first sight.
The latter imposes substantial restrictions on the existence of uniform $q$-contact structures for $q$ greater than $2$, hinting that such foliations might not occur so frequently.
It is also our first result specifically concerning contact \emph{foliations}, in that it does not apply for contact flows, i.e., for the classic case.
It is a result nicely related to a Theorem of Goertsches and Loiudice \cite{goertsches_how_2020} stating that every metric $f$-K-contact manifold can be constructed from a K-contact manifold utilising mapping tori and the so-called ``type II'' deformations (cf. \cite[Theorem 4.4]{goertsches_how_2020}).
\par We also introduced two possible generalisations for the Weinstein conjecture in the context of contact foliation. We were able to show that the strongest of such generalisations holds true for Anosov and isometric contact actions. In particular, the results in the isometric case expand results previously known in Contact Dynamics literature, due to Rukimbira and Banyaga \cite{rukimbira_remarks_1993, banyaga_characteristics_1995}.
\par For future research, in an upcoming paper we aim to investigate the geometrical properties of quasiconformal contact foliations, as well as the existence of closed leaves, effectively weakening the hypothesis of Theorem \ref{closedorbitsriemanniancase}. This theory could also be applied to prove classification theorems form quasiconformal Anosov bundles, generalising established results from Fang \cite{fang_smooth_2004}.

\providecommand{\bysame}{\leavevmode\hbox to3em{\hrulefill}\thinspace}
\providecommand{\MR}{\relax\ifhmode\unskip\space\fi MR }
\providecommand{\MRhref}[2]{%
  \href{http://www.ams.org/mathscinet-getitem?mr=#1}{#2}
}
\providecommand{\href}[2]{#2}

\end{document}